\documentclass{amsart} 

\theoremstyle{remark} 
\newtheorem{theorem}{Theorem}
\newtheorem{lemma}{Lemma}

\newtheorem{corollary}{Corollary}
\newtheorem{remark}{Remark}
\newtheorem{definition}{Definition}
\newtheorem{example}{Example}
 
\newtheorem{assumption}{Assumption}

\numberwithin{equation}{section}

\usepackage{graphicx} 
\usepackage{extarrows}
\usepackage{latexsym}
\usepackage{amsmath}
\usepackage{amssymb}
\usepackage{amsfonts}
\usepackage{verbatim}
\usepackage{mathrsfs}
\usepackage{color}
\usepackage{xcolor}
\usepackage[colorlinks,citecolor=blue,urlcolor=blue]{hyperref}
\usepackage{soul}

\begin{document}

\title[Well-posedness, Regularity, and Strong Approximations of Superlinear SRDE]
    {Well-posedness, Regularity, and Strong Approximations of Superlinear Stochastic Reaction-Diffusion Equation}
    
\author{Zhihui LIU}
\address{Department of Mathematics \& National Center for Applied Mathematics Shenzhen (NCAMS) \& Shenzhen International Center for Mathematics, Southern University of Science and Technology, Shenzhen, 518055, P. R. China}
% \curraddr{}
\email{liuzh3@sustech.edu.cn} 

\thanks{The author is supported by the National Natural Science Foundation of China (NNSFC), No. 12671474, Basic and Applied Basic Research Foundation of Guangdong Province, No. 2024A1515012348, and Shenzhen Basic Research Special Project (Natural Science Foundation) Basic Research (General Project), No. JCYJ20240813094919026.}

%    General info
\subjclass[2020]{Primary 60H35; Secondary 60H15, 65M60}

\date{}
 
\keywords{Stochastic reaction-diffusion equation,
stochastic Allen--Cahn equation,
superlinear diffusion coefficient,
uniform-in-time strong convergence rate, 
tamed scheme}

\begin{abstract}
This paper develops a general framework for the well-posedness, regularity, and strong approximation of the stochastic reaction--diffusion equation (SRDE) with superlinear drift and diffusion coefficients. We first extend the well-posedness results in \emph{W. Liu and M. R\"ockner, J. Funct. Anal., 2902--2922, 2010} and \emph{W. Liu, J. Differential Equations, 572--592, 2013} to the case of superlinear diffusion in the Gelfand triple \(V \hookrightarrow H \hookrightarrow V^*\), with \(V\) equipped with the norm \(\|\cdot\|_V\), and derive a moment estimate by establishing a new It\^o formula for \(\|X\|_V^p\) with general \(p \ge 2\). We then apply this abstract result to the SRDE, establish higher spatial regularity \(\dot H^{1+\gamma}\) for any \(\gamma \in [0,1]\) whenever the initial datum lies in the same Sobolev space, and obtain temporal H\"older regularity. Finally, we construct a family of tamed finite element methods (tamed-FEMs) for the SRDE under general assumptions on the tamed functions, derive their long-time unconditional stability, and establish optimal strong convergence rates. To our knowledge, this is the first strong approximation result for SPDEs with superlinear diffusion coefficients.
\end{abstract}

\maketitle

\section{Introduction}\label{sec1}

Consider the SPDE 
\begin{align} \label{spde} 
\,\mathrm{d}X(t,\xi)=[\Delta X(t,\xi)+ f(X(t,\xi))] \,\mathrm{d}t + g(X(t,\xi))\,\mathrm{d}W(t,\xi), 
& \quad t>0,~ \xi \in \mathcal O,    
\end{align} 
driven by an infinite-D \(Q\)-Wiener process \(W\) with the (homogeneous) Dirichlet boundary condition (DBC) \(X(t,\xi)=0\), \((t, \xi)\in [0, \infty) \times \partial \mathcal O\).
Here \(\mathcal O \subset \mathbb{R}^d\), \(d=1,2,3\), is a bounded, open domain with piecewise smooth boundary \(\partial \mathcal O\), \(f, g: \mathbb R \to \mathbb R\) are measurable functions which grow polynomially and satisfy certain monotone-type conditions, and the initial datum \(X(0, \cdot)=X_0(\cdot)\) is a Hilbert space valued random variable.
In modeling reaction-diffusion, one classical choice of the drift function is \(f(\xi)=-|\xi|^p \xi\), \(\xi \in \mathbb R\), for general \(p>0\); another is \(f(\xi)=\xi-\xi^3\), \(\xi \in \mathbb R\), corresponding to the stochastic Allen--Cahn equation (SACE), arising from phase transition in materials science by stochastic perturbation \cite{Fun16}.

\subsection{Numerical Motivations}
\label{sec1.1}

There exists a general theory of strong convergence analysis for full discretizations of SPDEs with Lipschitz coefficients; we refer to \cite{ACLW16, Bre13, CHL17, CHL18, JKW11, JR15, KW19, WGT14} and the references therein.
In the past two decades, researchers have investigated strong approximations of SPDEs with non-Lipschitz coefficients using monotone-type conditions; see, e.g., \cite{CH19, FLZ17, GM09, HS23, Liu26, LQ20, LQ21, LS26, Wan20} for the stochastic Allen--Cahn equation (SACE), \cite{ABNP21, QCW24, QW20} for the stochastic Cahn--Hilliard equation, \cite{Bou18, BMPW26, BP24} for stochastic Navier--Stokes equations, and \cite{BC22, CHL17b, CHLZ17, CHLZ19} for the stochastic nonlinear Schr\"odinger equations.

The authors in \cite{BHJKLS19} showed that the classical Euler--Maruyama (EM) scheme and its Galerkin-based full discretizations, when applied to Eq. \eqref{spde} with superlinear drift (driven by additive white noise), lead to blow-up in the \(p\)-th moment for all \(p \geq 2\).
This, in particular, indicates that EM is not appropriate for temporally discretizing Eq. \eqref{spde} with superlinear growth coefficients.

To overcome this difficulty, researchers have primarily pursued two distinct approaches.
The first is to apply an implicit scheme, including the backward EM scheme or its modification (see, e.g., \cite{FLZ17, GM09, HS23, LQ20, LQ21}) and the stochastic theta method in \cite{LL25}.
The main ingredient is that these implicit schemes are more stable than EM in the sense that they can inherit the monotonicity structure of the original SPDE.

On the other hand, schemes that treat the nonlinearity explicitly are often preferred to reduce computational costs. For SACE driven by additive noise, the authors in \cite{BGJK23, BJ19} constructed a truncated explicit scheme, combined with spatial Galerkin approximations, that truncates the nonlinearity by a factor \(\mathbf{1}_{\Omega_\tau}\), where \(\Omega_\tau\) denotes a set for which the numerical solution at the underlying step is sufficiently small.
This judgment-based truncation strategy was then generalized to a tamed argument in \cite{Wan20}, where the author tamed the nonlinearity by a factor \((1+\tau \|\mathcal P_N F(Y_m^N)\|)^{-1}\), and the norm of the underlying numerical solution was used.

It is known that, to construct such explicit schemes that inherit the monotonicity structure and strongly converge to the solution of Eq. \eqref{spde}, unconditional stability and the related Lyapunov structure play key roles.
Most recently, the author and Shen in \cite{LS26} first constructed a family of Galerkin-based explicit-in-nonlinearity tamed schemes, called tamed-FEMs, for a general class of Eq. \eqref{spde} with superlinear drift and Lipschitz diffusion driven by multiplicative noise. These tamed-FEMs are more efficient to implement because they tame the nodal values rather than their norms at each step.
They derived a Lyapunov estimate for \eqref{t-fem} with a specified choice of tamed operators, which established the unconditional stability of \(\{Y_n^h\}\) over an infinite time horizon, and established sharp strong convergence rates over any finite time horizon.

After the construction in \cite{LS26}, many research groups have focused on tamed methods for superlinear SPDEs; see \cite{CL26, CLW26, JW25, WC26}.
For example, these strong convergence rates were generalized in \cite{CL26} to be uniform in time, provided that a certain strong dissipativity condition holds; see also \cite{CLW26} for uniform-in-time weak error and ergodic estimates.
Whether such a Lyapunov estimate and strong convergence rates for the tamed-FEM applied to \eqref{spde} with superlinear diffusion remain valid is still open.
Among others, one of the main difficulties comes from the well-posedness and regularity, especially the moment estimates, of Eq. \eqref{spde} with superlinear diffusion.

\subsection{Mathematical Motivations}

Eq. \eqref{spde} can be rewritten as an infinite-D SDE of the form
\begin{equation}\label{see}
\,\mathrm{d}X(t)=D(X(t))\,\mathrm{d}t+G(X(t))\,\mathrm{d}W(t), \quad t>0; \quad X(0)=X_0,
\end{equation}
which admits a unique solution in the variational framework (a Gelfand triple \(V\hookrightarrow H\hookrightarrow V^*\) exists) if the corresponding Nemytskii-type drift and diffusion operators satisfy the classical monotone and coercivity conditions (cf. \cite{LR15}). In recent years, this variational approach has also been used extensively in mathematical and numerical analysis; see, e.g., \cite{KR79, Par75, RRW07} for general well-posedness results, \cite{GM09} for numerical approximation schemes, and \cite{Wan07} for the dimension-free Harnack inequality and resulting ergodicity properties of the associated transition semigroups.

As pointed out in \cite[Remark 1.1 (3)]{LR10}, if one considers an SRDE of the form \eqref{spde}, i.e., \(D(u) = \Delta u + f(u)\), then for verifying the classical growth condition \(\|D(v)\|_{V^*} \le C (1 + \|v\|_V^{\alpha-1})\) with some \(\alpha>1\) (see \cite[\textup{(H4)} in p. 70]{LR15}) one has \(\alpha = 2\). Hence, this condition would imply that \(f\) has at most linear growth.
It was then generalized in \cite{LR10} to some polynomial growth drift satisfying
\begin{equation}\label{lr}
\|D(v)\|_{V^*} \le C(1+\|v\|_V^{\alpha-1}) (1+\|v\|_H^\beta),
\end{equation}
for some \(\beta >0\); see \cite[\textup{(H4)}]{LR10}.
We note that in the 2-D or 3-D case, there is a restriction on the polynomial order \(q+1\) in \cite[Example 3.2]{LR10}: \(q<2\).
Indeed, for Eq. \eqref{spde} or the SPDE given in \cite[Example 3.2]{LR10} with Lipschitz diffusion, with \(V=\dot H^1\), \(H=L_\xi^2\), and \(\alpha=q+1\), it is not difficult to show that \eqref{lr} does not hold for \(q \ge 2\) and \(d=2,3\).

To establish the well-posedness of Eq. \eqref{see} with Lipschitz diffusion in the space \(V\), \cite{Liu13} proposed a Lyapunov-type (one-sided linear growth) condition \(\langle D(v), v \rangle_{V} \le C (1 + \|v\|_V^2)\) on the drift, which holds on an orthogonal set in \(V\) that constitutes an orthonormal basis of \(H\).
They conjectured that the solution is right continuous in \(V\); see \cite[Remark 1.2(2)]{Liu13}.
We also note that, for Eq. \eqref{spde} with linear drift and superlinear diffusion of the form \(\gamma_0 |X|^{1+\gamma}\) with a constant \(\gamma>0\), the authors in \cite{CH21, Mue91} used the semigroup (or equivalent Green function) method.
Their strategy for handling the nonlinearity of the diffusion was to obtain a sharp estimate for the general Lipschitz case and apply it to the superlinear case. In the variational framework, to our knowledge, there is no study on SPDEs with superlinear diffusion coefficients.

Besides the numerical requirement discussed in Section \ref{sec1.1}, this provides another motivation for studying the well-posedness of SPDEs with superlinear drift and diffusion coefficients. Instead of the previous two types of growth conditions, we add a term involving a power of the \(\|\cdot\|_{V_0}\)-norm in \eqref{D-grow} for another Banach space \(V_0\) (say, $L_\xi^{q+2}$) with \(V \subset V_0\), which removes the restriction on \(q\) and allows polynomial growth diffusion; see \textup{(H4)} in Section \ref{sec2}.
In particular, we propose a family of application-friendly conditions (see Assumption \ref{ap}) on the functions \(f\) and \(g\) in Eq. \eqref{spde}. To our knowledge, conditions of this type have not been proposed before.

\subsection{Main Results}

Our well-posedness and moment estimate results in Theorems~\ref{tm-well} and \ref{tm-mom} give an affirmative answer to the continuity conjecture in \cite{Liu13}, and generalize those in \cite{LR10} and \cite{Liu13} to the polynomial growth diffusion case.

Back to numerical aspects, we also propose a family of general conditions (see \eqref{coe-tau}) such that (s.t.) the tamed-FEMs \eqref{t-fem} are unconditionally stable (see Theorem \ref{tm-lya}).
Moreover, provided that certain growth conditions (see Assumptions \ref{ap-tau}-\ref{ap-tau'}) for the tamed functions and their derivatives hold, we prove that the associated auxiliary process \eqref{aux} converges towards the solution to Eq. \eqref{spde} with sharp strong convergence rates in \(1,2,3\)-D cases.
All these conditions on the tamed functions are comparable to, and more delicate than, the finite-D ones proposed by the author and Wu in \cite{LWWZ25} (see Remark \ref{rk-ex}).

To handle the error between the auxiliary process and the tamed-FEM approximation, we establish certain Sobolev and H\"older regularities of the exact solution and the uniform Sobolev regularity of the auxiliary process (see Theorem \ref{tm-reg+} and Lemma \ref{tm-aux}, resp.), and use a variational decomposition of the error term different from that in \cite{CL26, LS26}. Under the current conditions, the auxiliary process is shown to enjoy Sobolev regularity in the 1-D case (see Remark \ref{rk-aux}). Hence, the tamed-FEM \eqref{t-fem} achieves optimal strong convergence rates for Eq. \eqref{spde} in the 1-D case, and these rates are uniform in time under certain dissipativity conditions (the same as those in \cite{LS26} for the Lipschitz diffusion case). To our knowledge, this is the first strong approximation result for SPDEs with superlinear diffusion coefficients. 

The paper is organized as follows. In the rest of Section \ref{sec1}, we give frequently used notations and conventions. In Section \ref{sec2}, we establish the well-posedness and moment estimates of monotone SDEs with superlinear drift and diffusion operators. These results are generalized and applied to the SRDE \eqref{spde} in Section \ref{sec3}, where we propose explicit conditions on its coefficients. Finally, in Section \ref{sec4}, we construct a family of tamed-FEMs for the SRDE under general assumptions on the tamed functions, derive their long-time unconditional stability, and establish optimal strong convergence rates.

\subsection{Notations}

Let \((\Omega,\mathcal F,\mathbb P)\) be a complete probability space with a normal filtration \(\mathbb F=\{\mathcal F_t\}_{t\ge0}\).
For \(n\in\mathbb N\), denote by \(\mathbb E_n[\cdot]:=\mathbb E[\cdot\mid\mathcal F_{t_n}]\) the conditional expectation with respect to \(\mathcal F_{t_n}\).

For a constant \(p\in(1,\infty)\), let \(p^*\) denote its conjugate exponent, i.e., \(1/p+1/p^*=1\), and let \(p^+\) and \(p^-\) denote numbers sufficiently close to \(p\) from above and below, resp. Let \((L_\xi^p,\|\cdot\|_{L_\xi^p})\), \((L^p_t,\|\cdot\|_{L^p_t})\), and \((L_\omega^p,\|\cdot\|_{L_\omega^p})\) be the usual real-valued Lebesgue spaces on \(\mathcal O\), \([0,T]\) with \(T>0\) fixed, and the sample space \(\Omega\), resp.
For convenience, we frequently use the temporal, sample-path, and spatial mixed norms \(\|\cdot\|_{L^p_\omega L^{p_1}_t L_\xi^{p_2}}\) in different orders.
We use \(\rightharpoonup\) and \(\overset{*}\rightharpoonup\) to denote weak and weak-star convergences, resp., in various Banach spaces.
We also denote by \(\mathcal C_t^\gamma L_\omega^p \dot H^\beta\) the space of \(\mathbb F\)-adapted processes \(Z\) s.t.
$\|Z\|_{\mathcal C_t^{\gamma} L_\omega^p \dot H^\beta}
:= \sup_{0\le s < t <\infty} \|Z(t)-Z(s)\|_{L_\omega^p \dot H^\beta} |t-s|^{-\gamma}<\infty$.
We omit the integration variable when an integration is present to lighten the notation.

Throughout, we use \(C\), \(c\), \(c_1\), etc., and \(\epsilon\), \(\epsilon_1\), etc., to denote generic positive constants and sufficiently small positive constants, resp., independent of $T$ and various discrete parameters, which may differ from one occurrence to another.

\section{Well-posedness and Moment Estimates of Monotone SDEs with Superlinear Drift and Diffusion Operators}
\label{sec2}

\subsection{Preliminaries} 
\label{sec2.1}

Let \((H, \|\cdot\|, \langle\cdot,\cdot\rangle)\) be a separable Hilbert space, identified with its dual \(H^*\) via the Riesz isomorphism, and let \((V, \|\cdot\|_V, \langle\cdot,\cdot\rangle_V)\) be another Hilbert space continuously and densely embedded in \(H\). Then we obtain the Gelfand triple $V\hookrightarrow H\hookrightarrow V^*$, where \(V^*\), with norm \(\|\cdot\|_{V^*}\), is the dual of \(V\) with respect to \(\langle\cdot,\cdot\rangle\), and the dual pairing \(\langle\cdot,\cdot\rangle_{V^*,V}\) satisfies
$\langle u,v\rangle_{V^*,V}=\langle u,v\rangle$ for all $u\in H$ and $v\in V$.

Let \(V_0\) be a Banach space continuously and densely embedded in \(H\) s.t. \(V\hookrightarrow V_0\). Then we obtain another Gelfand triple
$V_0\hookrightarrow H\hookrightarrow V_0^*$, 
where \(V_0^*\), with norm \(\|\cdot\|_{V_0^*}\), is the dual of \(V_0\) with respect to \(\langle\cdot,\cdot\rangle\), and
$\langle u,v\rangle_{V_0^*,V_0}=\langle u,v\rangle$ for all $u\in H$ and $v\in V_0$.
Moreover, since \(V\hookrightarrow V_0\),
$\langle u,v\rangle_{V_0^*,V_0}=\langle u,v\rangle_{V^*,V}$ for all $u\in V_0^*$ and $v\in V$.

Let \((U,\|\cdot\|_U,\langle\cdot,\cdot\rangle_U)\) be a separable Hilbert space, and \(Q\) be a self-adjoint, positive semidefinite, trace-class operator on \(U\) with eigenpairs \(\{(q_m,g_m)\}_{m=1}^\infty\), where \(\{g_m\}_{m=1}^\infty\) is an orthonormal basis of \(U\). Set
$U_0:=Q^{1/2}U$, with norm \(\|\cdot\|_{U_0}\) and inner product $\langle u,v\rangle_{U_0}:=\langle Q^{-1/2}u,Q^{-1/2}v\rangle_U$, $u,v\in U_0$.
Denote by
$(\mathcal L_2^0:=HS(U_0;H),\|\cdot\|_{\mathcal L_2^0})$ and
$(\mathcal L_2^1:=HS(U_0;V),\|\cdot\|_{\mathcal L_2^1})$
the spaces of Hilbert--Schmidt operators from \(U_0\) to \(H\) and \(V\), resp. 
Let \(W=\{W(t):t\ge0\}\) be a \(U\)-valued \(Q\)-Wiener process on \((\Omega, \mathcal F, \mathbb F, \mathbb P)\); that is, there exist mutually independent 1-D Brownian motions \(\{\beta_m\}_{m=1}^\infty\) s.t.
$W=\sum_{m=1}^\infty \sqrt{q_m}\,g_m\beta_m$.
Equivalently,
$W=\sum_{m=1}^\infty \hat g_m\beta_m$, with
$\hat g_m:=\sqrt{q_m}\,g_m$, $m\in\mathbb N_+$, 
so \(W\) is also a \(U_0\)-cylindrical Wiener process. 

Let
$D_1: V \to V^*$, $D_0: V_0 \to V_0^*$, and $G: V \to \mathcal L_2^0$
be measurable operators (progressively measurable if they depend on \((t,\omega)\)) and set \(D=D_1+D_0\). 
We impose the following conditions on \(D\) and \(G\), where \(p\ge 2\) and \(\alpha_0\ge 2\) are constants. 
\begin{itemize} 
\item[\textup{(H0)}] ({\bf Orthogonality}) There exists an orthogonal set $\{e_{1},e_2,\cdots \}$ in $V$ s.t. it constitute an orthonormal basis of $H$.

    \item[\textup{(H1)}] ({\bf Hemicontinuity}) For all $u,v,w\in V$ and $i=1,2$, $\lambda \langle  D_i(u+\lambda v), w\rangle_{V^*, V}$ is continuous on $\mathbb R$.
    
    \item[\textup{(H2)}] ({\bf Monotonicity}) There exists a constant $c_0$ s.t. for all $u, v \in V$,
    \begin{align} \label{mon}
   & 2 \langle D(u)-D(v), u-v\rangle_{V^*, V} 
   + (p-1) \|G(u)-G(v)\|_{\mathcal L_2^0}^2 
   \le c_0 \|u-v\|^2.
    \end{align}     
    
    \item[\textup{(H3)}] ({\bf Coercivity}) For any $n \in \mathbb{N}_+$, the operator $D$ maps $H^n:=\{e_k\}_{k=1}^n$ into $V$, and there exist constants $c_1, c_2 \in \mathbb R$ and $c_3>0$ s.t. for all $v \in H^n$,
\begin{align} \label{coe}
2 \langle D(v), v \rangle_{V} + (p-1) \|G(v)\|_{\mathcal L_2^1}^2 
\le c_1+ c_2 \| v \|_{V}^2 - c_3 \| v \|_{V_0}^{\alpha_0}.
\end{align}

    \item[\textup{(H4)}] ({\bf Growth}) There exist positive constants $c_4$, $c_5$, and $c_6$ s.t. for all $v \in V$,
    \begin{align} 
    \|D_1(v)\|_{V^*}  \le c_4 (1 + \|v\|_V), & \quad  
\|D_0(v)\|_{V_0^*} \le c_5 (1 + \|v\|_{V_0}^{\alpha_0-1}), \label{D-grow} \\
\|G(v)\|_{\mathcal L_2^1}^2 
  \le & c_6 (1 + \|v\|_V^2 + \|v\|_{V_0}^{\alpha_0}). \label{G-grow}  
    \end{align}    
\end{itemize}

Although our methods and results can naturally be extended to \((t,\omega)\)-dependent coefficients, for simplicity we assume throughout that all operators and coefficients are independent of \((t,\omega)\). We also note that the results in Section~\ref{sec2} apply to the case where the exponent \(2\) on \(\|v\|_V\) is replaced by a constant \(\alpha>1\); the SRDE \eqref{spde} corresponds to \(\alpha=2\).

For brevity, we set
\begin{align*}
I:=\mathcal C_t L_\omega^2 V, \quad 
& J_0:=L_{t, \omega}^2 \mathcal L_2^0, \quad 
J:=L_{t, \omega}^2 \mathcal L_2^1, \\ 
K:=L_{t, \omega}^2 V, \quad
K^*=L_{t, \omega}^2 V^*, \quad 
& K_0:=L_{t, \omega}^{\alpha_0} V_0, \quad 
K_0^*=L_{t, \omega}^{\alpha_0^*} V_0^*.
\end{align*}

\begin{definition}
A continuous \(V\)-valued \(\mathbb{F}\)-adapted process \(X=(X(t))_{t\in[0,T]}\) is called a (variational) solution of Eq.~\eqref{see} if \(X\in I\cap K\cap K_0\) and ($\mathbb P$-a.s.)
\begin{align*}
X(t)=X_0 + \int_0^t D(X)\,\mathrm{d}s+\int_0^t  G(X)\,\mathrm{d}W,\quad t \in [0, T].
\end{align*} 
By pathwise uniqueness of the solution, we mean that, for any other solution \(Y\) of Eq.~\eqref{see} with the same initial datum \(Y(0)=X_0\),
$\mathbb P(X(t)=Y(t), \forall ~ t\in[0,T])=1.$
\end{definition}

\subsection{Existence of Solution in $V$}
\label{sec2.2}

By \textup{(H0)}, \(\{e_k\}_{k=1}^\infty\) is an orthogonal set in \(V\) and constitutes an orthonormal basis of \(H\).
For each \(n\in\mathbb N\), let \(H_n:=\operatorname{span}\{e_k\}_{k=1}^n\) and let \(P_n:V^*\to H_n\) be the generalized orthogonal projection defined by
\begin{align*}
P_n y:=\sum_{i=1}^n \langle  y, e_i \rangle_{V^*, V} e_i,\quad y\in V^*.
\end{align*}
Then for all $y, z \in V^*$ and $w \in V$,
\begin{align*}
\langle  z,P_n y \rangle_{V^*, V} = \langle  y, P_n z \rangle_{V^*, V}, \quad 
\langle  P_n y, w \rangle_{V^*, V} = \langle  y, P_n w \rangle_{V^*, V}.
\end{align*}
Let $U_0^n:=\text{Span} \{{\hat g}_{1}, \cdots, {\hat g}_{n}\}$ and $\hat{P}_n: U_0 \to U_0^n$ be the orthogonal projection operator, and set 
\begin{align*}
W^n(t):=\hat{P}_n W(t)=\sum_{i=1}^n \langle  W(t), {\hat g}_i\rangle_{U_0} {\hat g}_i, \quad n \in \mathbb N.
\end{align*} 
Then \(W^n\) is an \(n\)-D \(U_0^n\)-valued Wiener process on \((\Omega,\mathcal F,\mathbb F,\mathbb P)\).

For each \(n\in\mathbb N\), we consider the following \(n\)-D SDE on \((H_n,\|\cdot\|_V,\langle\cdot,\cdot\rangle_V)\), driven by the \(n\)-D \(U_0^n\)-valued Wiener process \(W^n\), with \(X^n_0:=P_nX_0\):
\begin{align} \label{Xn}
\,\mathrm{d}X^n(t)=P_n D(X^n(t))\,\mathrm{d}t+P_nG(X^n(t))\,\mathrm{d}W^n(t), 
~ t \in [0, T].  
\end{align}
By finite-D SDE theory, under \textup{(H1)}--\textup{(H3)}, Eq.~\eqref{Xn} has a unique \(H_n\)-valued $\mathbb F$-adapted continuous strong solution \(X^n\) s.t. 
$\mathbb E\sup_{t\in[0,T]}\|X^n(t)\|_V^2<\infty$, see, e.g., \cite[Theorem 3.1.1]{LR15}.

Below, we provide uniform a priori estimates for the solution to Eq.~\eqref{Xn}, which are needed to derive the existence of a solution to Eq.~\eqref{see} in \(V\).

\begin{lemma}\label{lm-Xn}
Let \(p\ge2\) and assume that \textup{(H0)}--\textup{(H4)} hold and \(X_0\in L^p(\Omega;V)\). Then there exists a positive constant \(C\) s.t.
\begin{align}\label{est-Xn}
& \sup_{t \in [0, T]} \mathbb E \|X_{t}^{n}\|_{V}^{p} + \mathbb E \int_{0}^T \|X^n\|_{V}^{p-2} \|X^n\|_{V_0}^{\alpha_0} \,\mathrm{d}t \le e^{C T} (1+\mathbb E \|X_0\|_V^p),
\quad n \in \mathbb N_+. 
\end{align}
In particular,  
\begin{align} \label{est-Xn+} 
& \| X^n\| _K + \| X^n\| _{K_0} + \| D_1(X^n)\| _{K^*} + \| D_0(X^n)\| _{K_0^*} + \| G(X^n)\| _J \nonumber \\
& \le e^{C T} (1+\mathbb E \|X_0\|_V^p),
\quad n \in \mathbb N_+. 
\end{align}
\end{lemma}

\begin{proof}
By \textup{(H0)}, we have
$\langle e_k,v \rangle_{V_i} = 0$, for $k \ge  n+1$ and $v\in H_{n}$.
By \textup{(H3)}, we know that $D(u)\in V \subset H$ for $u\in H_{n}$ and thus
\begin{align*}
\sum_{i = 1}^{\infty} {}_{V^{*}}\langle D(u),e_{i}\rangle_{V} e_{i} 
= \sum_{i = 1}^{\infty} \langle D(u),e_{i}\rangle e_{i} = D(u).
\end{align*}
Then, for any $u,v\in H_{n}$, we have
\begin{align}\label{PnV} 
\langle P_{n}D(u), v \rangle_{V}
&= \sum_{i = 1}^{n}  \langle D(u),e_{i}\rangle_{V^{*}, V} \langle e_{i}, v \rangle_V
= \sum_{i = 1}^\infty  \langle D(u),e_{i}\rangle_{V^{*}, V} \langle e_{i}, v \rangle_V \nonumber \\
& =  \langle \sum_{i = 1}^{\infty} \langle D(u),e_{i}\rangle_{V^{*}, V} e_{i}, v  \rangle_V
= \langle D(u), v \rangle_{V}.
\end{align} 

Let $t \in [0, T]$, $p \ge 2$, and $0 \le n \le N$.
Using finite-D It\^{o} formula, we have  
\begin{align} \label{ito-n}
\|X^n(t)\|_V^{p}  
& = \|X^n_0\|_V^{p}  + M^n(t) + \frac{p(p-2)}2 \int_{0}^{t} \|X^n\|_V^{p-4}
   \|[P_n G(X^n) \hat{P}_{n}]^* X^n\|_{U_0}^2 \,\mathrm{d}t  \nonumber \\
&\quad + \frac p2 \int_{0}^{t}\|X^n\|_V^{p-2}
[2 \langle P_n D(X^n), X^n \rangle_V 
  + \|P_n G(X^n)\hat{P}_{n}\|_{\mathcal L_2^1}^2] \,\mathrm{d}t, 
\end{align} 
where  
$M^n(\cdot):= p \int_0^\cdot \|X^n\|_V^{p-2} \langle  X^n, P_n G(X^n) \,\mathrm{d}W^n(s) \rangle_V$.
For each $R \in \mathbb N_+$, define 
$$\tau_R^n:=\inf\{t\ge0:\|X^n(t)\|_V \ge R\}\wedge T.$$
Then $\{\tau_R^n\}$ is a sequence of $\mathbb F$-stopping times with $\tau_R^n\uparrow T$ as $R \to \infty$ s.t. for each $R \in \mathbb N_+$, $\|X^n(t\wedge\tau_R^n)(\omega)\|_V$ is bounded uniformly in $(t,\omega)\in[0, T]\times\Omega$ and that $\{M^n(t\wedge\tau_R^n):~ t \in [0, T]\}$ is a martingale.
Consequently, $M^n$ is a local martingale with localizing sequence $\{\tau_R^n\}$ s.t. $\mathbb E M^n(t)=0$.

Applying \eqref{PnV}, \eqref{coe}, and Young inequality, we have 
\begin{align*}
& \mathbb E \|X^n(t\wedge\tau_R^n)\|_V^{p} - \mathbb E \|X_0^n\|_V^{p} \nonumber \\
&\le \frac p2 \mathbb E \int_{0}^{t} \mathbf{1}_{[0, \tau_R^n]}  \|X^n\|_{V}^{p-2}
[2 \langle D(X^n), X^n\rangle_{V}
  + (p-1) \|G(X^n)\|_{\mathcal L_2^1}^2] \,\mathrm{d}t \\
& \le \frac p2 \mathbb E  \int_{0}^{t} \mathbf{1}_{[0, \tau_R^n]}  \|X^n\|_{V}^{p-2} (c_1 + c_2\|X^n\|_V^2 - c_3 \|X^n\|_{V_0}^{\alpha_0}) \,\mathrm{d}t \\
&\le \mathbb E \int_0^t \mathbf{1}_{[0,\tau_R^n]} (c_1' + c_2' \|X^n\|_V^p -  c_3' \|X^n\|_V^{p-2} \|X^n\|_{V_0}^{\alpha_0}) \,\mathrm{d}s,
\end{align*}
for some constants $c_1'$, $c_2'$, and $c_3'=p c_3$.    
It follows that 
\begin{align*}
&\mathbb E \|X^n(t \wedge \tau_R^n)\|^p + c_3' \mathbb E \int_0^t \mathbf{1}_{[0,\tau_R^n]} \|X^n\|_V^{p-2} \|X^n\|_{V_0}^{\alpha_0} \,\mathrm{d}s \\
&\le \mathbb E \|X_0^n\|^p + c_1' t + c_2' \mathbb E \int_0^t \mathbf{1}_{[0,\tau_R^n]} \|X^n\|_V^p \,\mathrm{d}s.
\end{align*}
By the definition of $\{\tau_R^n\}$, all terms are finite. Now taking $R \to \infty$ and applying Fatou lemma, Young inequality, and the fact that $\mathbb E \|X^n_0\|_V \le \mathbb E \|X_0\|_V$, we get
\begin{align} \label{Xn-lp}
& \mathbb E \|X^n(t)\|_V^p + c_3' \mathbb E \int_0^t \|X^n\|_V^{p-2} \|X^n\|_{V_0}^{\alpha_0} \,\mathrm{d}s
\le \mathbb E \|X_0\|_V^p + c_1' t + c_2' \int_0^t \mathbb E \|X^n\|_V^p \,\mathrm{d}s.
\end{align}  
Then we derive \eqref{est-Xn} and the first two summands in \eqref{est-Xn+} with $p=2$. 

For the remaining summands in \eqref{est-Xn+}, the assertion then follows from \eqref{D-grow}-\eqref{G-grow} with $\alpha_0 \ge 2$, and Young inequality:
\begin{align*}
\|D_1(X^n)\|_{K^*}^2 
% & = E\int_0^T \|D(X^n)\|_{V^*}^2 \,\mathrm{d}t
& \le C \mathbb E \int_0^T (1 + \|X^n\|_V)^2 \,\mathrm{d}t
\le C (T + \|X^n\|_K^2 ),  \\
\|D_0(X^n)\|_{K_0^*}^{\alpha_0^*} 
% & = E\int_0^T \|D(X^n)\|_{V^*}^{\alpha_0^*}\,\mathrm{d}t
& \le C\mathbb E \int_0^T (1 + \|X^n\|_{V_0}^{\alpha_0-1})^{\alpha_0^*} \,\mathrm{d}t 
\le C (T + \|X^n\|_{K_0}^{\alpha_0}), \nonumber  \\
\|G(X^n)\|_J^2  
& \le C E\int_0^T (1 + \|X^n\|_V^2 + \|X^n\|_{V_0}^{\alpha_0})\,\mathrm{d}t 
\le C (T + \|X^n\|_K^2+\|X^n\|_{K_0}^{\alpha_0}).
\end{align*} 
All of them are bounded, since the first two terms in \eqref{est-Xn+} are bounded.
\end{proof}

As a consequence of the uniform estimate \eqref{est-Xn}, we obtain the following existence result for solutions to Eq.~\eqref{see} in \(V\).

\begin{theorem}  \label{tm-well}
Let $p \ge 2$, $X_0 \in L^p(\Omega; V)$, and \textup{(H0)}-\textup{(H4)} hold. Eq. \eqref{see} exists a solution $X \in L_t^\infty L_\omega^p V \cap L_{t, \omega}^{\alpha_0} V_0$ which is a Markov process.
\end{theorem}

\begin{proof}
By the reflexivity of \(K\), \(K_0\), \(K^*\), \(K_0^*\), and \(J\), together with the uniform estimate \eqref{est-Xn} and the Banach--Alaoglu and Kakutani theorems, there exists a subsequence \(\{n_k\} \subset\mathbb N\) with \(n_k\to\infty\) s.t.  
\begin{enumerate}
\item $X^{n_k}\rightharpoonup \bar X$ in $K$ and $K_0$ and $X^{n_k} \overset{*}\rightharpoonup \bar X$ in $L_t^\infty L_\omega^p V$;

\item $Y_1^{n_k}:=P_{n_k} D_1(X^{n_k}) \rightharpoonup Y_1$ and $N^{n_k}:=\int_0^\cdot Y_1^{n_k} \,\mathrm{d}s \rightharpoonup N_1:=\int_0^\cdot Y_1 \,\mathrm{d}s$ both in $K^*$, and $Y_2^{n_k}:=P_{n_k} D_0(X^{n_k}) \rightharpoonup Y_2$ and $N_2^{n_k}:=\int_0^\cdot Y_2^{n_k} \,\mathrm{d}s \rightharpoonup N_2:=\int_0^\cdot Y_2 \,\mathrm{d}s$ both in $K_0^*$; and  

\item $Z^{n_k}:=P_{n_k} G(X^{n_k}) \rightharpoonup Z$ in $J$ and hence
$\int_0^\cdot Z^{n_k} \,\mathrm{d}W^{n_k}(s)
%= \int_0^\cdot P_{n_k} G(X^{n_k})\widetilde P_{n_k}\,\mathrm{d}W 
\overset{*}\rightharpoonup
\int_0^\cdot Z \,\mathrm{d}W$
in $L_t^\infty L_\omega^2 H$.
\end{enumerate}

Now, \(\mathrm dt\times\mathbb P\)-a.e., \(\bar X=X_0+N+M=:X\), where \(N:=N_1+N_2\). The remainder of the existence proof and the proof of the Markov property can be adapted from \cite[Theorem 1.1]{Liu13} and \cite[Proposition 4.3.3]{LR15}, resp. 
\end{proof}

\subsection{It\^o formula and Moment Estimates}
\label{sec2.3}

In this section, we derive a new It\^o formula for \(\|X\|_V^p\), for a general exponent \(p\ge2\), to estimate moments of solutions to Eq.~\eqref{see}. This formula is crucial in both the mathematical and numerical analysis of SPDEs, especially in the analysis of strong convergence rates in Section~\ref{sec4}.

In general, one cannot prove the a.s. convergence of the Galerkin approximations \eqref{Xn} in \(\mathcal C([0,T];V)\), nor can one show that the solution of Eq.~\eqref{see}, as their limit, is continuous in \(V\) and satisfies an analogue of the It\^o formula \eqref{ito-n}.  
One obstacle is that the monotone condition \eqref{mon} is assumed on \(H\) rather than on \(V\), and such a \(V\)-version does not hold in many commonly used situations.  
Another deterministic compactness argument, namely the Aubin--Lions argument, is also inapplicable, since it requires the temporal derivative of the solution to lie in a certain Banach space, which is never the case in the stochastic setting.

To overcome these difficulties, we assume the existence of a symmetric positive definite linear operator \(A: V_2:=\operatorname{Dom}(A)\subset H\to H\) (e.g., the negative Dirichlet Laplacian \(-\Delta\)) s.t. \(V=\operatorname{Dom}(A^{1/2})\).
In this case, the norm and inner product on \( V \) can be defined, resp., as
\begin{align} \label{df-V}
\|x\|_V^2 = \|A^{1/2} x\|^2 = \langle A x, x \rangle, \quad 
\langle x, y \rangle_V = \langle A^{1/2} x, A^{1/2} y \rangle, 
\quad x, y \in V. 
\end{align}  

As $X$ satisfies Eq. \eqref{see}, we have   
\begin{align*}
d [A^{1/2} X(t)] = A^{1/2} D(X(t))\,\mathrm{d}t + A^{1/2} G(X(t))\,\mathrm{d}W(t), ~ 0<t \le T. 
\end{align*} 
Since we have proved that \(X\in L_t^\infty L_\omega^p V\) for \(p\ge2\), it is natural to set $U:=A^{1/2} X \in L_t^\infty L_\omega^p H$ and reduce the desired It\^o formula in \(V\) to one in \(H\).

In the following, we first derive an It\^o formula for the process
\begin{align} \label{ito-}
X(t):=X_0+ \int_0^t Y \,\mathrm{d}s + \int_0^t Z \,\mathrm{d}W, \quad t \ge 0,
\end{align} 
where \(Y\) and \(Z\) are general progressively measurable processes, so as to include more concrete examples. The proof for the case \(p=2\) follows the idea in \cite{RRW07}. We then prove the general case \(p\ge2\) by using the chain rule for a 1-D It\^o process and martingale representation theory. Finally, we extend the formula to the space \(V\).

\begin{lemma}  \label{lm-ito}
Let $p \ge 2$ and assume that $X$ satisfies \eqref{ito-} with $X_0\in L^p(\Omega; H)$, $X\in K \cap K_0$, $Y \in K^* + K_0^*$, and $Z \in J_0$. Then $X$ is an $H$-valued continuous $\mathbb F$-adapted process s.t. $E \sup_{t \in [0, T]}\|X(t)\|^p<\infty$ and 
\begin{align} \label{ito}
\frac1p \|X(t)\|^{p} 
&= \frac1p \|X_0\|^{p} + \int_{0}^{t}\|X\|^{p-2}
  \langle X, Z  \,\mathrm{d}W \rangle
+ \frac{p-2}2 \int_{0}^{t}\|X\|^{p-4}
   \|Z^* X\|_{U_0}^2 \,\mathrm{d}t  \nonumber \\
&\quad + \frac12 \int_{0}^{t}\|X\|^{p-2}
[2 \langle Y, X \rangle_{V^*, V} + \|Z\|_{\mathcal L_2^0}^2] \,\mathrm{d}t, \quad t \in [0, T].
\end{align}  
\end{lemma}

\begin{proof} 
Set $Y=Y_1+Y_2$ with $Y_1 \in K^*$ and $Y_2 \in K_0^*$. 

(1) {\bf The case $p=2$.} 
Set $M=\int_0^\cdot Z \,\mathrm{d}W$.
For all $0 \le s < t \le T$ s.t. $X(t), X(s)\in V$, from $X(t)=X_s+ \int_s^t Y \,\mathrm{d}s + M(t)-M(s)$ it is not difficulty to show that 
\begin{align} \label{X2} 
&\|M(t)-M(s)\|^2 - \|X(t)-X(s)-M(t)+M(s)\|^2 + 2\langle  X(s), M(t)-M(s)\rangle\\
% & = - \|X(t)-X(s)\|^2 + 2 \langle X(t)-X(s), M(t)-M(s) \rangle + 2\langle  X(s), M(t)-M(s)\rangle  \nonumber \\
% & = - \|X(t)-X(s)\|^2 + 2 \langle X(t), M(t)-M(s) \rangle  \nonumber\\
% & =- \|X(t)\|^2 - \|X(s)\|^2 + 2 \langle X(t), X(s) \rangle + 2 \langle X(t), M(t)-M(s) \rangle  \nonumber\\
% & =- \|X(t)\|^2 - \|X(s)\|^2 + 2 \langle X(t), X(s)+ M(t)-M(s) \rangle  \nonumber\\
% & =- \|X(t)\|^2 - \|X(s)\|^2 + 2 \langle X(t), X(t)-\int_s^t Y \,\mathrm{d}s \rangle_{V^*, V} \nonumber \\
& = \|X(t)\|^2 - \|X(s)\|^2 - 2\int_s^t \langle Y, X(t) \rangle_{V^*, V} \,\mathrm{d}r.   \nonumber 
\end{align}

Let $\{I_l: ~ l\in\mathbb N\}$ be a sequence of partitions (as in \cite[Lemma 4.2.6 and Remark 4.2.7]{LR15}) s.t. $I_l\subset I_{l+1}$ and $\delta(I_l):=\max_i(t_i^l-t_{i-1}^l)\to0$ as $l\to\infty$ s.t. $X(t_i^l)\in V \cap V_0$ $\mathbb P$-a.e. for all $l\in\mathbb N,1\le i\le k_l-1$, and 
$\bar X^l:=\sum_{i=2}^{k_l}\mathbf{1}_{[t_{i-1}^l,t_i^l)}X(t_{i-1}^l), ~ \widetilde X^l:=\sum_{i=1}^{k_l-1}\mathbf{1}_{[t_{i-1}^l,t_i^l)}X(t_i^l) \in K \cap K_0$ s.t. 
\begin{align} \label{X-app}
\lim_{l\to\infty} \|X-\bar X^l\|_K
+ \lim_{l\to\infty} \|X-\bar X^l\|_{K_0}
= \lim_{l\to\infty} \|X-\widetilde X^l\|_K 
+ \lim_{l\to\infty} \|X-\widetilde X^l\|_{K_0} =0.
\end{align}
By the assumption that $X\in K \cap K_0$, we may choose $I_l$ s.t. 
$\mathbb E \|X(t)\|^2<\infty \text{ for all }t \in \cup_{l\in\mathbb N}I_l$, thus $E \sup_{1\le j\le K-1}\|X(t_j^l)\|^2 < \infty$.
Then by \eqref{X2}, for any $t=t_i^l\in I_l\setminus\{0, T\}$,
\begin{align} \label{ito-dis} 
& \|X(t)\|^2 - \|X_0\|^2 
= \sum_{j=0}^{i-1} (\|X(t_{j+1}^l)\|^2 - \|X(t_j^l)\|^2) \\
&= 2\int_0^t \langle  Y, \widetilde X^l\rangle_{V^*, V} \,\mathrm{d}s + 2\int_0^t \langle  \bar X^l, Z\,\mathrm{d}W\rangle + 2 \int_0^{t_1^l} \langle X_0, Z\,\mathrm{d}W\rangle \nonumber \\
&\quad + \sum_{j=0}^{i-1} [\|M(t_{j+1}^l)-M(t_j^l)\|^2 - \|X(t_{j+1}^l)-X(t_j^l)-M(t_{j+1}^l)+M(t_j^l)\|^2]. \nonumber 
\end{align}
We note that since $\bar X^l$ is $\mathbb F$-adapted and pathwise bounded, the stochastic integral involving $\bar X^l$ above is well-defined. 
Moreover, as $\lim_{l\to\infty} \|X-\widetilde X^l\|_K =0$, 
\begin{align*}
\mathbb E \int_0^T |\langle  Y, \widetilde X^l\rangle_{V^*, V}| \,\mathrm{d}s 
& \le \mathbb E \int_0^T |\langle  Y_1, \widetilde X^l\rangle_{V^*, V}| \,\mathrm{d}s 
+ \mathbb E \int_0^T |\langle  Y_2, \widetilde X^l\rangle_{V^*, V}| \,\mathrm{d}s \\
& \le \|\widetilde X^l\|_K \|Y_1\|_{K^*} + \|\widetilde X^l\|_{K_0} \|Y_2\|_{K_0^*} 
 \le C (\|Y_1\|_{K^*} + \|Y_2\|_{K_0^*}),
\end{align*}
where $C$ is a constant depending on $\|X\|_K$ and $\|X\|_{K_0}$, but independent of $l$.
Moreover, by Burkholder--Davis--Gundy (BDG) and Young inequalities, we have 
\begin{align*}
E \sup_{t\in[0, T]}  \Big|\int_0^t \langle  \bar X^l, Z \,\mathrm{d}W\rangle \Big| 
&\le 3E \Big(\int_0^T \|Z^*\bar X^l\|_U^2 \,\mathrm{d}s \Big)^{1/2}
\le 3E \Big( \int_0^T \|\bar X^l\|^2 \|Z\|_{\mathcal L_2^0}^2 \,\mathrm{d}s \Big)^{1/2} \nonumber \\ 
&\le \frac{1}{4}E \sup_{1\le j\le K-1}\|X(t_j^l)\|^2 
+ 9 \|Z\|_J^2.  
\end{align*}    
For the first summation on the RHS of \eqref{ito-dis}, we have 
\begin{align*}
E \sum_{j=0}^{i-1}\|M(t_{j+1}^i)-M(t_j^i)\|^2
% &= \sum_{j=0}^{i-1}E\|\int_{t_j^i}^{t_{j+1}^i}Z(s)\,\mathrm{d}W\|^2\\
&= \mathbb E \int_0^{t_i^i}\|Z(s)\|_{\mathcal L_2^0}^2 \,\mathrm{d}s
= E \langle  M\rangle_{t_i^i} \le \|Z\|_J^2.  
\end{align*} 

Combining the above three estimates with \eqref{ito-dis} and Young inequality, we obtain
\begin{align*}
E \sup_{t\in I_1\setminus\{T\}}\|X(t)\|^2 
\le C_1:=C (\mathbb E \|X_0\|^2 + \|Y_1\|_{K^*} + \|Y_2\|_{K_0^*} + \|Z\|_J^2)<\infty.
\end{align*}  
Therefore, letting $l\uparrow\infty$ and setting $I:=\cup_{l\ge1}I_l\setminus\{T\}$, we obtain
$E \sup_{t\in I}\|X(t)\|^2 \le C_1$, as $I_l\subset I_{l+1}$ for all $l\in\mathbb N$. Since $t\mapsto\|X(t)\|$ is lower semicontinuous $\mathbb P$-a.s. and $I$ is dense in $[0, T]$, we have $\sup_{t\in[0, T]}\|X(t)\|^2 = \sup_{t\in I}\|X(t)\|^2$, 
and thus $E \sup_{t\in[0, T]}\|X(t)\|^2 <\infty$.
Therefore, $X$ is $H$-valued $\mathbb P$-a.s. and by its continuity in $V^*$ it is weakly continuous in $H$ $\mathbb P$-a.s.
Moreover, $X$ is progressively measurable as an $H$-valued process, since $\mathcal{B}(H)$ is generated by $H^*$. 
By the proof of the It\^o formula in \cite[Theorem 4.2.5]{LR15}, we have  
\begin{align} \label{x-xl}
\lim_{l\to\infty}\sup_{t\in[0, T]}\left|\int_0^t \langle  X-\bar X^l, Z\,\mathrm{d}W\rangle\right| = 0 \text{ in } \mathbb P.  
\end{align}

Now, fix $0 \neq t\in I$. 
For each sufficiently large $l\in\mathbb N$, there exists a unique $0<i<k_l$ s.t. $t=t_i^l$. We have $X(t_j^l) \in V \cap V_0$ a.s. for all $j$. 
By \eqref{X-app},
\begin{align*}
& E \Big|\int_0^t \langle Y, \widetilde X^l-\bar X \rangle_{V^*, V} \,\mathrm{d}s \Big| \\
& \le E \int_0^t |\langle Y_1, \widetilde X^l-\bar X\rangle_{V^*, V}| \,\mathrm{d}s 
+ E \int_0^t |\langle Y_1, \widetilde X^l-\bar X\rangle_{V_0^*, V_0}|  \,\mathrm{d}s \\
& \le \|Y_1\|_{K^*} \|\widetilde X^l-\bar X\|_K
+ \|Y_2\|_{K_0^*} \|\widetilde X^l-\bar X\|_{K_0},
\end{align*}
which shows that $\int_0^t \langle  Y, \widetilde X^l-\bar X^l \rangle_{V^*, V} \,\mathrm{d}s\to0$ in $L^1(\Omega; \mathbb R)$, and thus in $P$, as $l \to \infty$.
This convergence, in combination with \eqref{x-xl}, yields that the sum of the first three terms in the right-hand side of \eqref{ito-dis} converges to
$2\int_0^t \langle  Y, \bar X \rangle_{V^*, V} \,\mathrm{d}s + 2\int_0^t \langle  X, Z\,\mathrm{d}W\rangle$  in $\mathbb P$, as $l \to \infty$. 
Hence, we have
\begin{align*}
\|X(t)\|^2 - \|X_0\|^2 = 2\int_0^t \langle  Y, \bar X\rangle_{V^*, V} \,\mathrm{d}s + 2\int_0^t \langle  X, Z\,\mathrm{d}W\rangle + \langle  M\rangle_t - \epsilon_0,
\end{align*}
where
$\epsilon_0:=\lim_{l\to\infty}\sum_{j=0}^{i-1}\|X(t_{j+1}^l)-X(t_j^l)-M(t_{j+1}^l)+M(t_j^l)\|^2$ (in $\mathbb P$) exists.
The arguments in \cite[Page 99]{LR15} show that $\epsilon_0=0$, and hence \eqref{ito} holds for every $0 \neq t \in I$. The extension of \eqref{ito} with $p=2$ to the case $t \in [0, T]$, together with the proof of the strong continuity of $t \mapsto \|X(t)\|^2$, is analogous to \cite[Pages 99--101]{LR15}; therefore, we omit the details.

(2) {\bf The case $p>2$.}
We now extend \eqref{ito} from $p=2$ to $p>2$.
Indeed, \eqref{ito} with $p=2$ yields that $\|X\|^2$ is an $\mathbb R$-valued It\^o process s.t.
\begin{align*}
\frac12 d \|X(t)\|^2 = \Big[ \langle Y, X\rangle_{V^*, V}
  + \frac12 \|Z\|_{\mathcal L_2^0}^2 \Big] \,\mathrm{d}t 
  + \langle X, Z  \,\mathrm{d}W \rangle, \quad t \in [0, T].
\end{align*}  

Since \(X\in L_\omega^2\mathcal C_tH\) and \(Z\in J\), it follows that $M(\cdot)=\int_0^\cdot \langle X, Z  \,\mathrm{d}W \rangle$ is a 1-D square-integrable martingale.
By the martingale representation theorem, there exists a 1-D predictable process $z$ satisfying $\|z\|_{L_{t, \omega}^2}<\infty$ and a 1-D Wiener process $B$ on $(\Omega, \mathcal F, \mathbb F, \mathbb P)$ s.t. 
$M(\cdot)=\int_0^\cdot z  \,\mathrm{d}B$ satisfying $z^2=\|Z^* X\|_{U_0}^2$.
Therefore, $R:=\|X\|^2$ is an It\^o process satisfying 
$d R(t) = y(t) \,\mathrm{d}t + z(t) \,\mathrm{d}B(t)$, where $y:= \langle Y, X\rangle_{V^*, V}
  + \frac12 \|Z\|_{\mathcal L_2^0}^2$ is a 1-D predictable process s.t. 
$\mathbb E \int_0^T |y(t)| \,\mathrm{d}t<\infty$.
  Using the chain rule, we obtain 
\begin{align*}
d |R(t)|^{\frac p2} 
= p |R(t)|^{\frac p2-2} [R(t) y(t) + z(t)^2] \,\mathrm{d}t + p |R(t)|^{\frac p2-2} R(t) z(t)  \,\mathrm{d}B(t).
\end{align*}
Then we conclude \eqref{ito} by substituting the representations of $R$, $y$, and $z$. 
\end{proof}

\begin{corollary} \label{cor-ito}
Let $p \ge 2$ and assume that $X$ satisfies \eqref{ito-} with $X_0 \in L^p(\Omega; V)$, $A^{1/2} X \in K \cap K_0$, $A^{1/2} Y \in K^* + K_0^*$, and $Z \in J$. Then $X$ is a $V$-valued continuous $\mathbb F$-adapted process s.t. $E \sup_{t \in [0, T]}\|X(t)\|_V^p<\infty$ and for all $t \in [0, T]$,  
\begin{align} \label{ito-V}
& \frac1p \|X(t)\|_V^{p} 
= \frac1p \|X_0\|_V^{p} 
+ \frac{p-2}2 \int_{0}^{t}\|X\|_V^{p-4}
   \|[A^{1/2} Z]^* [A^{1/2} X]\|_V^2 \,\mathrm{d}t  \\
&\quad + \frac12 \int_{0}^{t}\|X\|_V^{p-2}
[2 \langle A^{1/2} Y, A^{1/2} X \rangle_{V^*, V}
  + \|Z\|_{\mathcal L_2^1}^2] \,\mathrm{d}t + \int_{0}^{t}\|X\|_V^{p-2}
  \langle X, Z \,\mathrm{d}W \rangle_V. \nonumber 
\end{align}  
\end{corollary}

\begin{proof} 
Let $t \in [0, T]$.
It is clear that ${\widetilde X}:=A^{1/2} X$ satisfies
$d {\widetilde X}(t) = {\widetilde Y}(t) \,\mathrm{d}t + {\widetilde Z}(t) \,\mathrm{d}W(t)$,  
with ${\widetilde Y}:=A^{1/2} Y$ and ${\widetilde Z} := A^{1/2} Z$.  

As $A^{1/2} X \in K \cap K_0$, $A^{1/2} Y \in K^* + K_0^*$, and $Z\in J$, it is clear that ${\widetilde X} \in K \cap K_0$, ${\widetilde Y} \in K^* + K_0^*$, and ${\widetilde Z} \in J_0$. 
This shows that the conditions in Lemma \ref{lm-ito} hold s.t. \eqref{ito} holds with $X$, $Y$, $Z$ replaced by ${\widetilde X}$, ${\widetilde Y}$, ${\widetilde Z}$ defind above:
\begin{align*}
 \frac1p \|A^{1/2} X(t)\|^{p}  
&= \frac1p \|A^{1/2} X_0\|^{p} + \int_{0}^{t}\|A^{1/2} X\|^{p-2}
  \langle A^{1/2} X, A^{1/2} Z  \,\mathrm{d}W \rangle  \nonumber \\
&\quad 
+ \frac{p-2}2 \int_{0}^{t}\|A^{1/2} X\|^{p-4}
   \|[A^{1/2} Z]^* [A^{1/2} X]\|_{U_0}^2 \,\mathrm{d}t  \nonumber \\
&\quad + \frac12 \int_{0}^{t} \|A^{1/2} X\|^{p-2}
[2 \langle A^{1/2} Y, A^{1/2} X \rangle_{V^*, V}
  + \|A^{1/2} Z\|_{\mathcal L_2^1}^2] \,\mathrm{d}t.
\end{align*}  
This coincides with \eqref{ito-V}, taking into account the definition \eqref{df-V} of the norm and inner product on $V$ in the present setting.
\end{proof}

Now we are at the position to derive moment estimates of Eq. \eqref{see}.

\begin{theorem}  \label{tm-mom}
Let $p \ge 2$ and Assume that $X_0 \in L^p(\Omega; V)$ and \textup{(H0)}-\textup{(H4)} hold.
    Then the solution to Eq. \eqref{see} is unique, satisfying 
\begin{align} \label{sta}
\mathbb E \|X(t)-Y(t)\|^{p_0} \le e^{p_0 c_0 t/2} \mathbb E \|X_0-Y_0\|^{p_0}, \quad \forall ~ t \in [0, T],
\end{align}  
where $Y$ denotes the solution of Eq. \eqref{see} starting from $Y(0)=Y_0$.
Assume furthermore that $A^{1/2} X \in K \cap K_0$ and $A^{1/2} D(X) \in K^* + K_0^*$, then there exists a positive constant $C$ s.t.  
\begin{align}\label{est-X}
& \mathbb E \|X(t)\|_V^p + \mathbb E \int_0^t \|X(r)\|_V^{p-2} \|X(r)\|_{V_0}^{\alpha_0} \,\mathrm{d}s
\le e^{C t} (1+\mathbb E \|X_0\|_V^p).
\end{align}  
Moreover, if $c_2<0$, then
\begin{align} \label{est-X+} 
& \sup_{t \ge 0} \mathbb E \|X(t)\|_V^{p}  
+ \sup_{t \ge 1} \Big\{\frac1t \int_{0}^{t} \mathbb E \|X(r)\|_{V_0}^{\alpha_0} \,\mathrm{d}t \Big\}
\le C ( 1 + \mathbb E \|X_{0}\|_V^{p}). 
\end{align}  
\end{theorem}

\begin{proof}  
By Theorem \ref{tm-well}, $X \in L_t^\infty L_\omega^p V \cap K_0 \subset K \cap K_0$.
Due to \eqref{D-grow} and \eqref{G-grow}, $Y=D_1(X) + D_0(X) \in K^* + K_0^*$, and $Z=G(X) \in J_1$, so the conditions in Lemma \ref{lm-ito} hold.
By the It\^o formula \eqref{ito} in $H$ applied to $\|\cdot\|^p$ and \eqref{mon}, we have   
\begin{align*}
& \mathbb E \|X(t)-Y(t)\|^{p_0}  \\
& \le \mathbb E \|X_0-Y_0\|^{p_0} + \frac {p_0} 2 \mathbb E \int_{0}^{t}\|X-Y\|^{p_0-2}
[2 \langle D(X)-D(Y), X-Y \rangle_{V^*, V} \\
& \qquad + (p_0-1) \|G(X)-G(Y)\|_{\mathcal L_2^0}^2] \,\mathrm{d}t \\
& \le \mathbb E \|X_0-Y_0\|^{p_0} + \frac {p_0 c_0} 2 \mathbb E \int_{0}^{t} \|X-Y\|^{p_0} \,\mathrm{d}t.
\end{align*}   
From Gronwall lemma, we have the stability result \eqref{sta}, which shows the uniqueness of the solution with $X_0=Y_0$: $X(t)=Y(t)$, $\mathbb P\text{-a.s.}$, for all $t \in [0, T]$. Moreover, pathwise uniqueness follows from the pathwise continuity of $ X$ and $Y$ in $H$.

If $A^{1/2} X \in K \cap K_0$, $A^{1/2} D(X) \in K^* + K_0^*$, then the conditions in Corollary \ref{cor-ito} hold.
Applying the It\^o formula \eqref{ito} in $V$ to $\|\cdot\|_V^p$ and using the stopping time argument, as in the proof of \eqref{Xn-lp}, we obtain  
\begin{align} \label{X-lp}
& \mathbb E \|X(t)\|_V^p + c_3' \int_0^t \mathbb E [ \|X\|_V^{p-2} \|X\|_{V_0}^{\alpha_0}] \,\mathrm{d}s
 \le \mathbb E \|X_0\|_V^p + c_1' t + c_2' \int_0^t \mathbb E \|X\|_V^p \,\mathrm{d}s.
\end{align}    
By Gronwall lemma, we get \eqref{est-X}.
 
Finally, to prove the uniform-in-time estimate \eqref{est-X+} in the case $c_2<0$, we note that, by Young inequality, the constant $c_2'$ in \eqref{X-lp} can be chosen as $c_2'=- (- p c_2/2-\epsilon)$ for $\epsilon <-p c_2/2$.
Then, using the product rule, we derive
\begin{align*}
& \mathbb E [e^{\gamma t}\|X(t)\|_V^p] 
 + (- p c_2/2-\gamma-\epsilon) \mathbb E \int_{0}^{t} e^{\gamma s} \|X(s)\|_V^p \,\mathrm{d}t  \nonumber \\
& \quad + c_3' \mathbb E \int_{0}^{t} e^{\gamma s} \|X(s)\|_V^{p-2}  \|X(s)\|_{V_0}^{\alpha_0} \,\mathrm{d}t 
\le \mathbb E \|X_0\|_V^p + c_1' \int_0^t e^{\gamma s}  \,\mathrm{d}s,
\quad \forall ~ \gamma \in \Big[0, -\frac{p c_2}2 \Big).
\end{align*} 
Multiplying $e^{-\gamma t}$ on both hand sides above and chosing $\epsilon <-p c_2/2-\gamma$, we obtain    
\begin{align*}
& \mathbb E \|X(t)\|_V^{p} + (- p c_2/2-\gamma-\epsilon) \mathbb E \int_{0}^{t} e^{-\gamma (t-s)} \|X\|_V^p \,\mathrm{d}t   \nonumber \\
& \quad + c_3' \mathbb E \int_{0}^{t} e^{-\gamma (t-s)} \|X\|_V^{p-2}  \|X\|_{V_0}^{\alpha_0} \,\mathrm{d}t
\le e^{-\gamma t} \mathbb E \|X_{0}\|_V^{p} + c_1' \int_{0}^{t} e^{-\gamma s} \,\mathrm{d}s.
\end{align*} 
Then we conclude \eqref{est-X+} (for the second summand in \eqref{est-X+}, we take $p=2$ and $\gamma=0$ on the above inequality).
\end{proof}

Finally, we establish higher regularity under the following stronger but more useful conditions, which replace \textup{(H3)} and \textup{(H4)}. For ease of application, we omit the spaces $K_0$ and $K_0^*$ here.

\begin{theorem}  \label{tm-H2} 
Assume that $X_0 \in L^p(\Omega; V)$ and \textup{(H0)}-\textup{(H4)} hold, with \textup{(H3)} and \textup{(H4)} replaced, resp., by the following assumptions:
\begin{itemize}  
\item[\textup{(H3')}] For any $n \in \mathbb{N}_+$, the operator $D$ maps $H^n:=\{e_k\}_{k=1}^n \subset V_2$ into $V$, and there exist constants $\hat c_1, \hat c_3 \in \mathbb R$ and $\hat c_2, \hat c_4>0$ s.t. for all $v \in H^n$,
\begin{align} \label{coe+} 
2 \langle D(v), v \rangle_{V} + (p-1) \|G(v)\|_{\mathcal L_2^1}^2 
& \le \hat c_1- \hat c_2 \|A v\|^2 + \hat c_3 \|v\|_V^2 - \hat c_4 \| v \|_{V_0}^{\alpha_0}.
\end{align}

    \item[\textup{(H4')}] There exist positive constants $\hat c_5, \hat c_6$ and $\hat c_7$ s.t. for all $u \in V$ and $v \in V_2$,
    \begin{align} 
    \|D(u)\|_{V^*} \le \hat c_5 (1 + \|u\|_V + \|u\|_{V_0}^{\alpha_0-1}), & \quad 
    \|D(v)\| \le \hat c_6 (1 + \|A v\| + \|v\|_V^{\alpha_0-1}), \label{D-grow+} \\
\|G(v)\|_{\mathcal L_2^1}^2 
   \le \hat c_7 (&1 + \|A v\|^2 + \|v\|_{V_0}^{\alpha_0}). \label{G-grow+}  
    \end{align}     
\end{itemize}  
    Then the solution to Eq. \eqref{see} belongs to $L_{t, \omega}^2 V_2$ s.t.   
\begin{align}\label{est-X2}
& \mathbb E \|X(t)\|_V^p + \mathbb E \int_0^t \|A X\|^2 \,\mathrm{d}s
+ \mathbb E \int_0^t \|X\|_{V_0}^{\alpha_0} \,\mathrm{d}s
\le e^{C t} (1+\mathbb E \|X_0\|_V^p),
\end{align}  
for any $p \ge 2(\alpha_0-1)$.
Moreover, if $\hat c_3<0$, then 
\begin{align} \label{est-X2+} 
& \sup_{t \ge 0} \mathbb E \|X(t)\|_V^{p}  
+ \sup_{t \ge 1} \Big\{\frac1t \int_{0}^{t} \mathbb E \|A X\|^{\alpha_0} + \mathbb E \|X\|_{V_0}^{\alpha_0} \,\mathrm{d}t \Big\} 
\le C ( 1 + \mathbb E \|X_{0}\|_V^{p}). 
\end{align}  
\end{theorem}

\begin{proof} 
According to the proof of Lemma \ref{lm-Xn}, under the current condition, we have 
\begin{align*}
& \sup_{t \in [0, T]} \mathbb E \|X_{t}^{n}\|_{V}^{p} 
+ \mathbb E \int_{0}^T \|X^n\|_V^{p-2} \|A X^n\|^2 \,\mathrm{d}t
+ \mathbb E \int_{0}^T \|X^n\|_V^{p-2} \|X^n\|_{V_0}^{\alpha_0} \,\mathrm{d}t \\
& \le e^{Ct} (1+\mathbb E \|X_0\|_V^{p}). 
\end{align*}
By the proof of Theorem \ref{tm-well}, we get $X \in L_{t, \omega}^2 V_2$ and thus $A^{1/2} X \in K$, satisfying   
\begin{align*}
& \sup_{t \in [0, T]} \mathbb E \|X(t)\|_V^p 
+ \int_{0}^T \mathbb E \|A X\|^2 \,\mathrm{d}t
+ \int_{0}^T \mathbb E \|X^n\|_{V_0}^{\alpha_0} \,\mathrm{d}t 
<\infty.
\end{align*}       
Now, by the condition \eqref{D-grow+}, 
$\|D(X)\|_{L_{t, \omega}^2 H}
\le C(1 + \|X\|_{L_{t, \omega}^2 V_2} + \|X\|_{L_{t, \omega}^{2(\alpha_0-1)} V}^{\alpha_0-1})$, 
and the right-hand side is finite whenever $p \ge 2(\alpha_0-1)$. Hence, $D(X) \in L_{t, \omega}^2 H$.

Therefore, \(A^{1/2}X\in K\) and \(A^{1/2}D(X)\in K^*\).  
Consequently, applying the It\^o formula \eqref{ito} in \(V\) to \(\|X\|_V^p\), with general \(p\ge 2(\alpha_0-1)\), yields the estimates for the first terms in \eqref{est-X2} and \eqref{est-X2+}, resp.  
The estimates for the remaining two terms follow from the uniform estimates for the corresponding Galerkin approximations and from the weak lower semicontinuity of the norms in the spaces \(L_{t,\omega}^2V_2\) and \(L_{t,\omega,\xi}^{\alpha_0}\). 
\end{proof}

The above arguments can also be applied to establish the well-posedness of Eq.~\eqref{see} in \(V\) for a diffusion coefficient satisfying local monotonicity and polynomial growth, and in \(H\) for the case of coupled coercivity and polynomial growth.

\begin{remark}
Eq.~\eqref{see} has a unique solution \(X\in L_t^\infty L_\omega^pV\), provided that \(X_0\in L^p(\Omega;V)\) and \textup{(H0)}--\textup{(H4)} hold, where the monotonicity condition \eqref{mon} in \textup{(H2)} is replaced by the following local monotonicity condition in \(H\), valid for all \(u,v\in V\):
\begin{align*}
2\langle D(u)-D(v),u-v\rangle_{V^*,V}
+(p-1)\|G(u)-G(v)\|_{\mathcal L_2^0}^2
\le [c_0+\rho(v)]\|u-v\|^2,
\end{align*}
where \(\rho:V\to\mathbb R_+\) is a measurable function, and the growth condition
$|\rho(v)|\le c_8(1+\|v\|_V^2+\|v\|_{V_0}^{\alpha_0})$, $v \in V$, with \(c_8>0\), is added to \textup{(H4)}. We also note that in the linear-growth diffusion case, as in \cite{Liu13}, we can use the BDG inequality to show the \(L_t^\infty L_\omega^pV\)-estimate for the solution.
Therefore, the well-posedness result and moment estimates in Theorem~\ref{tm-well} give an affirmative answer to the continuity conjecture for \(\|X\|_V^p\) proposed in \cite[Remark 1.2(2)]{Liu13}, and generalize those in \cite[Theorem 1.1]{LR10} and \cite[Theorem 1.1]{Liu13} to the polynomial-growth diffusion case.
\end{remark}

\begin{remark}  \label{rk-H}
Eq. \eqref{see} possesses a unique solution in \(H\) satisfying \eqref{sta}--\eqref{est-X+}, with \(\|\cdot\|_V\) replaced by \(\|\cdot\|\), provided that \(X_0\in L^p(\Omega;H)\) with \(p\ge 2\), and that \textup{(H1)}--\textup{(H4)} hold, where \textup{(H3)} and \textup{(H4)} are replaced, resp., by the following coercivity and growth conditions in \(H\):
\begin{enumerate}
\item[\textup{(H3$^*$)}] There exist positive constants \(\tilde c_1,\tilde c_2,\tilde c_3\) s.t., for all \(v\in V\cap V_0\),
\begin{align*}
2\langle D(v),v\rangle_{V^*,V}
+(p-1)\|G(v)\|_{\mathcal L_2^0}^2
\le \tilde c_1+\tilde c_2\|v\|_V^2-\tilde c_3\|v\|_{V_0}^{\alpha_0}.
\end{align*}

\item[\textup{(H4$^*$)}] There exist positive constants \(\tilde c_4,\tilde c_5,\tilde c_6\) s.t., for all \(v\in V\cap V_0\),
\begin{align*} 
    \|D_1(v)\|_{V^*}  \le \tilde c_4 (1 + \|v\|_V), & \quad  
\|D_0(v)\|_{V_0^*} \le \tilde c_5 (1 + \|v\|_{V_0}^{\alpha_0-1}),  \\
 \|G(v)\|_{\mathcal L_2^0}^2  \le & \hat c_6 (1 + \|v\|_V^{2} + \|v\|_{V_0}^{\alpha_0}).  
    \end{align*} 
\end{enumerate}

We note that in this case the assumption \(V\subset V_0\) is not needed, since \textup{(H0)} is not used.
Moreover, the corresponding transition semigroup \((P_t)\) is Feller, with
$|P_t\phi(x)-P_t\phi(y)|
\le e^{c_0t/2} \|x-y\|$, $t \ge 0$, $x,y \in H$, 
for any \(1\)-Lipschitz function \(\phi\) in $H$.
Assume that \(V\) or \(V_0\) is compactly embedded in \(H\). Then \((P_t)\) admits an invariant measure on \(H\).
If \(c_0<0\), then \((P_t)\) has a unique invariant probability measure \(\mu\) with \(\mu(\|\cdot\|^2)<\infty\), and it converges exponentially fast to \(\mu\): 
$|P_t\phi(x)-\mu(\phi)|
\le e^{c_0t/2} \bigl(\mu(\|x-\cdot\|^2)\bigr)^{1/2}$, $t \ge 0$, $x,y \in H$, 
for any \(1\)-Lipschitz function \(\phi\) in $H$.
\end{remark}

\section{Well-posedness and Regularity of SRDE}
\label{sec3}

From this section on, let $q>0$ and choose $A=-\Delta$ (with DBC), $H=U:=\{u \in L_\xi^2:~ u|_{\partial \mathcal O}=0\}$ with norm $\|\cdot\|$, $V:=\dot H^1=\{u, \nabla u \in L_\xi^2:~ u|_{\partial \mathcal O}=0\}$ with norm $\|\cdot\|_1:=\|(-\Delta)^{1/2} \cdot\|$, $V_2:=\dot H^2=\{u, \nabla u, \Delta u \in L_\xi^2:~ u|_{\partial \mathcal O}=0\}$ with norm $\|\cdot\|_2:=\|\Delta \cdot\|$, and $V_0:=\{u \in L_\xi^{q+2}:~ u|_{\partial \mathcal O}=0\}$ with norm $\|\cdot\|_{L_\xi^{q+2}}$.

Then $V^*=\dot H^{-1}$, $V_0^*:=\{u \in L_\xi^{(q+2)^*}:~ u|_{\partial \mathcal O}=0\}$, $\dot H^1={\rm Dom}(A^{1/2})$, $\dot H^2={\rm Dom}(A)$, and the orthgonal condition \textup{(H0)} and all the settings in Section \ref{sec2} are valid.
Moreover, we have the following Poincar\'e inequality, with \(\lambda_1\) denoting the first eigenvalue of \(A\):
\begin{align}\label{poin}
\|u\|_1^2\ge \lambda_1\|u\|^2,\quad u\in\dot H^1;
\qquad
\|v\|_2^2\ge \lambda_1\|v\|_1^2,\quad v\in\dot H^2.
\end{align}

Throughout, we assume that $q>1$ when $d=1,2$ and $q \in (1, 3]$ when $d=3$, s.t. the following Sobolev embeddings hold:
\begin{align} \label{emb}
\dot H^1\hookrightarrow L_\xi^{2(q+1)} \hookrightarrow L_\xi^{q+2}
\hookrightarrow L_\xi^2 \hookrightarrow L_\xi^{(q+2)^*} \hookrightarrow L_\xi^{2(q+1)^*} \hookrightarrow \dot H^{-1},
\quad d=1,2,3, 
\end{align}
Then, we can define the Nemytskii operators $F: \dot H^1 \rightarrow \dot H^{-1}$ and $G: H \rightarrow \mathcal L_2^0$ associated with $f$ and $g$, resp., by
\begin{align} 
F(u)(\xi):= f(u(\xi)), & \quad u \in \dot H^1, ~ \xi \in \mathcal O, \label{df-F}\\
G(u) g_m(\xi):= g(u(\xi)) g_m(\xi), & \quad u \in H,~ m \in \mathbb N_+,~ \xi \in \mathcal O.
\label{df-G}
\end{align} 

With these preliminaries, Eq.~\eqref{spde} is equivalent to Eq.~\eqref{see} with
\[
D(u):=\Delta u+F(u),\qquad u\in\dot H^1,
\]
where \(F\) and \(G\) are defined in \eqref{df-F} and \eqref{df-G}, resp. 
 
Although our main results in this section apply to general SPDEs, including the stochastic \(p\)-Laplace equation, stochastic porous media equations, and fast diffusion equations studied in \cite{LR10, Liu13}, we focus primarily on the SRDE \eqref{spde}.

\subsection{Well-posedness of SRDE}
\label{sec3.1}

Since the diffusion coefficient is allowed to grow polynomially, we mainly focus on the trace-class case, i.e.,
$\sum_{m=1}^{\infty}q_m<\infty$.
For simplicity, we assume that the eigenvector \(g_m\) of \(Q\) is bounded for each \(m\in\mathbb N_+\) s.t. 
\begin{align} \label{cq}
C_Q:=\sum_{m \in \mathbb N_+} q_m \|g_m\|_{L_\xi^\infty}^2<\infty, \quad 
C_Q':=\sum_{m \in \mathbb N_+} q_m \|g'_m\|_{L_\xi^\infty}^2<\infty.
\end{align}

\begin{assumption} \label{ap}   
The functions $f, g: \mathbb R \to \mathbb R$ are continuously differentiable, and there exist constants $p \ge 2$, $L_0, L_1 \in \mathbb R$, $L_2 > 0$, and $L_3, L_4, L_5 \ge 0$, s.t. for all $\xi, \eta \in \mathbb R$,  
\begin{align} 
2 (\xi-\eta) [f(\xi)-f(\eta)]  + C_Q (p-&  1)  |g(\xi)-g(\eta)|^2
 \le L_0 |\xi-\eta|^2, \label{fg-mon} \\ 
f'(\xi) \le & L_1 - L_2 |\xi|^q, \label{f-coe+} \\ 
 |f'(\xi)| \le & L_3 (1+ |\xi|^q), \label{f-grow} \\ 
  |g'(\xi)|^2 \le & L_4 + L_5 |\xi|^{q/2}. \label{g-grow+}
\end{align}  
 \end{assumption}

The condition \eqref{f-coe+} is stronger than the following standard coercivity condition for some constants \(L_1'\in\mathbb R\) and \(L_2'>0\):
\begin{align}\label{f-coe}
\xi f(\xi)\le L_1'-L_2'|\xi|^{q+2},\qquad \xi\in\mathbb R.
\end{align}
A dual argument, in combination with the condition \eqref{f-grow}, implies that  
 \begin{align}  \label{F-} 
\|F(u)-F(v)\|_{-1} & \le C (1+\|u\|^q_1+\|v\|^q_1) \|u-v\|, 
\quad u, v \in \dot H^1.
\end{align}  
It follows from \eqref{fg-mon} and \eqref{f-coe} that the following polynomial growth conditions for \(g\) hold for some positive constants \(L_4'\) and \(L_5'\):
\begin{align}\label{g-grow-}
|g'(\xi)|^2\le L_4'(1+|\xi|^q),\qquad
|g(\xi)|^2\le L_5'(1+|\xi|^{q+2}),\qquad \xi\in\mathbb R.
\end{align}
However, we need the stronger growth condition \eqref{g-grow+}, which yields the following growth conditions for some positive constants \(\hat L_4\) and \(\hat L_5\), with
$\tilde g(\cdot):=\int_0^\cdot |g'(\eta)|^2\,d\eta$: 
\begin{align}\label{g-grow}
|\tilde g(\xi)|^2\le \hat L_4(1+|\xi|^{q+2}),\qquad
|g(\xi)|^2\le \hat L_5(1+|\xi|^{q/2+2}),\qquad \xi\in\mathbb R.
\end{align} 

We illustrate that if the functions \(f\) and \(g\) satisfy Assumption \ref{ap} with the stronger growth condition \eqref{g-grow+} replaced by \eqref{g-grow-}, then the growth condition \eqref{G-grow+} for the corresponding Nemytskii operator \(G\) will generally not be valid. Indeed, one then has to control the term
$\||v|^q\nabla v\|^2
\approx \||v|^{q+1}\Delta v\|_{L_\xi^1}$, 
which cannot be bounded by the right-hand side term
$c_5+c_6\|\Delta v\|_{L^2(\mathcal O)}^2
+c_7\|v\|_{L^{q+2}(\mathcal O)}^{q+2}$ 
in \eqref{G-grow+}. In the following result, we instead use the term
$\||v|^{q/2}\nabla v\|^2$ with a negative coefficient to bound
$\|v\|_{L^{q+2}(\mathcal O)}^{q+2}$
via the Poincar\'e inequality \eqref{poin}.

\iffalse
Then, \eqref{g-grow+} yields that  
{\color{blue}
 \begin{align*}
\|G(u)-G(v)\|_{\mathcal L_2^0}^2 
& \le K_0 C_Q \|(1+|u|^{q/2}+|v|^{q/2}) |u-v|^2\|_{L_\xi^1} \nonumber \\
& \le K_0 C_Q \|(1+|u|^{q/2}+|v|^{q/2})\|_{L_\xi^{(q+2)/(q/2)}} \| |u-v|^2\|_{L_\xi^{(q+2)/{(q/2+2)}}} \nonumber \\
& \le C (1+\|u\|_{L_\xi^{q+2}}^{q-2}+\|v\|_{L_\xi^{q+2}}^{q-2})\|u-v\|^2_{L_\xi^{4(q+2)/(q+4)}},  
\end{align*}  
This method can only derive order $<1/2$, as one needs to derive H\"older regularity of the solution to Eq. \eqref{spde} under the $L_\xi^{4(q+2)/(q+4)}$-norm. 
}
\fi

 \begin{theorem}  \label{tm-spde}
 For Eq.~\eqref{spde} with coefficients \(f\) and \(g\) satisfying Assumptions~\ref{ap}, the conditions \textup{(H0)}-\textup{(H2)}, \textup{(H3')}, and \textup{(H4')} hold with
$\alpha_0=q+2$, $c_0=L_0-2\lambda_1$, 
and some other positive constants \(\hat c_i\), \(i=1,\dots,7\).
Consequently, for any \(X_0\in L^p(\Omega;\dot H^1)\) with \(p\ge 2(q+1)\), Eq.~\eqref{spde} possesses a unique solution in
$L_\omega^p\mathcal C_t\dot H^1
\cap L_{t,\omega}^2\dot H^2
\cap L_{t,\omega,\xi}^{q+2}$, 
and there exists a constant $C$ s.t. 
\begin{align} \label{est-spde} 
& \|X\|^p_{L_t^\infty L_\omega^p \dot H^1}   
+ \|X\|_{L_{t, \omega}^2 \dot H^2}^{\alpha_0}  
+ \|X\|_{L_{t, \omega, \xi}^{q+2}}^{q+2} 
 \le e^{CT} ( 1 + \mathbb E \|X_{0}\|_1^{p}). 
\end{align}  
Moreover, if 
\begin{align} \label{con-uni}
(p-1)C_Q L_4 + \frac{[(p-1)C_Q L_5]^2}{8 L_2}< 2 (\lambda_1-L_1),
\end{align}  
then  
\begin{align} \label{est-spde+} 
& \sup_{t \ge 0} \mathbb E \|X(t)\|_1^{p}  
+ \sup_{t \ge 1} \Big\{\frac1t \int_{0}^{t} \mathbb E \|X\|_2^{\alpha_0} \,\mathrm{d}t \Big\}
+ \sup_{t \ge 1} \Big\{\frac1t \int_{0}^{t} \mathbb E \|X\|_{L_\xi^{q+2}}^{q+2} \,\mathrm{d}t \Big\} \nonumber \\
& \le C ( 1 + \mathbb E \|X_{0}\|_1^{p}). 
\end{align} 
\end{theorem}

\begin{proof}
The first two conditions \textup{(H0)} and \textup{(H1)} are standard.

By Poincar\'e inequality \eqref{poin}, \eqref{cq}, and \eqref{fg-mon}, we obtain 
\begin{align*}
& 2 \langle  D(u)-D(v), u-v \rangle_{V^*, V} + (p_0-1) \|G(u)-G(v)\|_{\mathcal L_2^0}^2
+ 2 \|\nabla (u-v)\|^2 +  \\
& = \int_{\mathcal O} [2 [f(u)-f(v)] (u-v) +  (p_0-1) \sum_{m \in \mathbb N_+} q_m |g_m|^2 |g(u)-g(v)|^2] \,\mathrm{d}\xi   \\
& = \int_{\mathcal O} [2 [f(u)-f(v)] (u-v) +  (2p-1) C_Q |g(u)-g(v)|^2] \,\mathrm{d}\xi 
\le L_0 |u-v|^2,
\end{align*}
which shows \eqref{mon} in \textup{(H2)} with $c_0=L_0-2 \lambda_1$. 
 
To show \eqref{coe+} and \eqref{G-grow+}, let us begin with an estimate of the term $\|g'(v) \nabla v \|^2$. 
Using integration by parts and the growth condition \eqref{g-grow} leads to  
\begin{align*}
\|g'(v) \nabla v \|^2 
& = \int_{\mathcal O} |g'(v)|^2 |\nabla v|^2 d \xi
= - \langle \tilde g(v), \Delta v \rangle
 \le \frac12 \|\Delta v\|^2 + \frac12 \|\tilde g(v)\|^2 \\
& \le C( 1 + \|\Delta v\|^2 + \|v\|_{L_\xi^{q+2}}^{q+2}).
\end{align*} 
The above estimate, in combination with \eqref{cq}, \eqref{g-grow}, and Young inequality, yields \eqref{G-grow+}with $\alpha_0=q+2$:
\begin{align} \label{proof-G-grow}
\|G(v)\|_{\mathcal L_2^1}^2 
 & = \int_{\mathcal O} \sum_{m \in \mathbb N_+} q_m \|[\nabla [g(v) g_m]\|^2 \,\mathrm{d}\xi \nonumber \\
  & \le 2 \int_{\mathcal O} \sum_{m \in \mathbb N_+} q_m \|[g(v)]' g_m\|^2 \,\mathrm{d}\xi + 2 \int_{\mathcal O} \sum_{m \in \mathbb N_+} q_m \|g(v) g_m'\|^2 \,\mathrm{d}\xi \nonumber \\
& \le 2 C_Q' \|g(v)\|^2 + 2 C_Q \|g'(v) \nabla v \|^2
\le C ( 1 + \|\Delta v\|^2 + \|v\|_{L_\xi^{q+2}}^{q+2}).
\end{align}  
In addition, by \eqref{f-grow}, Young inequality, and the embedding \eqref{emb},
\begin{align*} 
\|\Delta v + F (v)\|_{-1} & \le  \|v\|_1 +  C (1 + \|v\|_{L_\xi^{q+2}}^{q+1})
\le C (1 + \|v\|_1 + \|v\|_{L_\xi^{q+2}}^{q+1}), \\
\|\Delta v + F (v)\| & \le \|v\|_2 +  C (1 + \|v\|_{L_\xi^{2(q+1)}}^{q+1})
\le C (1 + \|v\|_2 + \|v\|_1^{q+1}), 
\end{align*}
so \eqref{D-grow+} in \textup{(H4')} holds.

Finally, by \eqref{cq}, \eqref{f-coe+}, and \eqref{g-grow+}, we have  
\begin{align*}
 & 2 \langle  D(v),v \rangle_V +  (p-1) \|G(v)\|_{\mathcal L_2^1}^2 \\ 
 & = -2 \|\Delta v\|^2 + \int_{\mathcal O} [ 2 f'(v) |\nabla v|^2 
  + (p-1) \sum_{m \in \mathbb N_+} q_m \|[\nabla [g(v) g_m]\|^2 ] \,\mathrm{d}\xi   \\
 % & \le -2 \|\Delta v\|^2 + \int_{\mathcal O} ( [2 L_1 -2 L_2 |v|^q+ (p-1)C_Q (1+\epsilon) L_4 (1 + |v|^{q/2})] |\nabla v|^2 + C_Q (p-1)  (1+\epsilon^{-1}) |g(v)|^2 ) \,\mathrm{d}\xi   \\
   & \le -2 \|\Delta v\|^2 + [2 L_1 + (p-1)C_Q (1+\epsilon) L_4 ] \|\nabla v\|^2 
  + C_Q (p-1)  (1+\epsilon^{-1}) \|g(v)\|^2   \\
  &  \quad - \int_{\mathcal O} [2 L_2 |v|^q - (p-1)C_Q (1+\epsilon) L_5 |v|^{q/2}] |\nabla v|^2 \,\mathrm{d}\xi.
    \end{align*}    
Then Young inequality leads to  
    \begin{align*}
 & 2 \langle  D(v),v \rangle_V +  (p-1) \|G(v)\|_{\mathcal L_2^1}^2 \\ 
& \le -2 \|\Delta v\|^2 + [2 L_1 + (p-1)C_Q (1+\epsilon) L_4
+ [(p-1)C_Q (1+\epsilon) L_5]^2 (4 \delta)^{-1} ] \|\nabla v\|^2  \\
  &  \quad - (2 L_2-\delta) \int_{\mathcal O} |v|^q |\nabla v|^2 \,\mathrm{d}\xi 
  + C_{\epsilon_1} + \epsilon_1 \|v\|_{L_\xi^{q+2}}^{q+2}.
    \end{align*} 
Applying the Poincar\'e inequality \eqref{poin} to \(w=|v|^{(q+2)/2}\in H_0^1(\mathcal O)\), we obtain  
\begin{align*}
\|v\|_{L_\xi^{q+2}}^{q+2}
= \| |v|^{(q+2)/2}\|_{L^2}^2
\le \lambda_1\|\nabla(|v|^{(q+2)/2})\|_{L^2}^2
=\frac{ (q+2)^2 \lambda_1}4 \int_{\mathcal O} |v|^q |\nabla v|^2\,d\xi, 
\end{align*}
and derive 
$\int_{\mathcal O} |v|^q |\nabla v|^2\,d\xi
\ge \frac 4{(q+2)^2\lambda_1} \|v\|_{L_\xi^{q+2}}^{q+2}$.
Thus,
    \begin{align*}
 & 2 \langle  D(v),v \rangle_V +  (p-1) \|G(v)\|_{\mathcal L_2^1}^2 \\ 
& \le -2 \|\Delta v\|^2 + [2 L_1 + (p-1)C_Q (1+\epsilon) L_4
+ [(p-1)C_Q (1+\epsilon) L_5]^2 (4 \delta)^{-1} ] \|\nabla v\|^2  \\
  &  \quad   + C_{\epsilon_1} - (4(2 L_2-\delta)[(q+2)^2\lambda_1]^{-1} 
- \epsilon_1 ) \|v\|_{L_\xi^{q+2}}^{q+2},
    \end{align*}     
so \eqref{coe+} in \textup{(H3')} holds with $\hat c_2=2$, $\hat c_3=2 L_1 + (p-1)C_Q (1+\epsilon) L_4 + [(p-1)C_Q (1+\epsilon) L_5]^2 (4 \delta)^{-1} $, and $\hat c_4=4(2 L_2-\delta)[(q+2)^2\lambda_1]^{-1} - \epsilon_1$.   
Here, \(\epsilon,\epsilon_1>0\) may be chosen arbitrarily small, and \(\delta>0\) may be chosen arbitrarily close to \(2L_2\) from below.

Finally, condition \eqref{con-uni}, together with the Poincar\'e inequality \eqref{poin}, ensures the existence of a positive constant \(\hat c_3\) (and another positive constant \(\hat c_4\)). This completes the verification of conditions \textup{(H0)}-\textup{(H2)}, \textup{(H3')}, and \textup{(H4')}. The last statement follows immediately from Theorem \ref{tm-H2}.
\end{proof}

 \begin{remark}
From the proof of Theorem \ref{tm-spde}, for Eq.~\eqref{spde} with coefficients \(f\) and \(g\) satisfying Assumption \ref{ap}, the conditions \textup{(H0)}--\textup{(H2)}, (H3*), and (H4*) hold with
$\alpha_0=q+2$, $c_0=L_0-2 \lambda_1$, $\tilde c_2=-2$, 
and other positive constants \(\tilde c_i\), \(i=1,3,4,5,6\).
Consequently, Eq.~\eqref{spde} possesses a unique solution in
$L_\omega^p \mathcal C_t H \cap L_{t, \omega}^2 \dot H^1 \cap L_{t, \omega, \xi}^{q+2}$, 
provided that \(X_0\in L^p(\Omega;H)\) with \(p\ge 2\), satisfying  
\begin{align*} 
& \sup_{t \ge 0} \mathbb E \|X(t)\|^{p} 
+ \sup_{t \ge 1} \Big\{\frac1t \int_{0}^{t} \mathbb E \|X\|_1^2 + \mathbb E \|X\|_{L_\xi^{q+2}}^{q+2} \,\mathrm{d}t \Big\} \le C ( 1 + \mathbb E \|X_{0}\|_1^{p}). 
\end{align*}
Moreover, the conclusions of Remark \ref{rk-H} remain valid. In this case, one can remove the embeddings \eqref{emb} -- and even \(\dot H^1\hookrightarrow L_\xi^{q+2}\) -- as well as the restriction \(p\ge2(q-1)\).
\end{remark}

\begin{remark}
The arguments in Theorem \ref{tm-spde} can be generalized to handle the gradient noise case: $G(v) g_m(\xi):=\gamma_0 \nabla v(\xi) g_m(\xi) + g(\xi) g_m(\xi)$, $m \in \mathbb N_+$, $\xi \in \mathcal O$, $\gamma_0>0$, with $g$ satisfying Assumption \ref{ap}. 
\end{remark}

\begin{remark}
Generally, \(\nabla X \notin L_{t,\omega,\xi}^{q+2}\), since we have just shown that $X \in L_{t,\omega}^2 \dot H^2 \cap L_t^\infty L_\omega^{q+2} \dot H^1 \cap L_{t,\omega,\xi}^{q+2}$ for $q>0$.
This is the main reason why we cannot derive $\nabla X \in K \cap K_0$ and why we omit the spaces $K_0$ and $K_0^*$. Consequently, additional restrictions such as $\dot H^1 \hookrightarrow L_\xi^{2(q+1)}$ and $p \ge 2(q+1)$ are needed in the proof of Theorem \ref{tm-spde}.
\end{remark}

 \begin{example}  \label{rk-ex} 
Assumption~\ref{ap} includes certain polynomial drifts (with negative leading coefficients) and diffusion functions. Indeed, let \(k\in\mathbb N\) and assume
 \begin{align} \label{ex}
f(\xi)=\sum_{i=0}^{2k+1} a_i \xi^i, \quad
g(\xi)=\sum_{j=0}^k c_j |\xi|^{j/2} \xi+c_0',
\quad \xi \in \mathbb R,
 \end{align} 
 where $a_i, c_j, c_0' \in \mathbb R$, $i=0,1,\cdots,2k+1$, $j=0,1,\cdots,k+1$, with $a_{2k+1}<0$. 
Clearly, \eqref{f-coe+}-\eqref{g-grow+} hold with \(q=2k\), \(L_2=-(2k+1)a_{2k+1}-\epsilon>0\) for an arbitrarily small \(\epsilon>0\), and certain constants \(L_1,L_3,L_4,L_5\). Moreover, elementary algebraic calculations yield \eqref{fg-mon} for a certain constant \(L_0\). 

These conditions are analogous to the finite-D conditions given in \cite{LWWZ25, LW26}. The only difference is the growth order of \(g\): it is \(1+q/4\) in the present setting, whereas in the finite-D case it can be \(1+q/2\) \cite[Assumption 3]{LW26}.
In particular, this includes the double-well potential drift \(f(\xi)=\xi-\xi^3\), \(\xi\in\mathbb R\), corresponding to SACE driven by trace-class noise satisfying \eqref{cq}, and the diffusion \(g(\xi)=\gamma_0|\xi|^{1/2}\xi\), where \(\gamma_0>0\) denotes the intensity parameter. In this case, we have \(L_1=1\), \(L_2=3\), \(L_4=0\), and \(L_5=9\gamma_0^2/4\), so the uniformity condition \eqref{con-uni} reduces to $27 C_Q^2 (p-1)^2 \gamma_0^4 < 256 (\lambda_1-1)$.
Similar arguments apply to the reaction-diffusion drift \(f(\xi)=-|\xi|^q \xi\) and diffusion \(g(\xi)=\gamma_0|\xi|^{q/4}\xi\) for the same range of \(q>0\) satisfying \eqref{emb}.

In the Lipschitz diffusion case, let \(L_g\) denote the Lipschitz constant of \(g\). Then one can choose \(L_5=0\) and \(L_4=L_g^2\) s.t. \eqref{con-uni} reduces to $C_Q (p-1) L_g^2 < 2 (\lambda_1- L_1)$. 
This is equivalent to the condition \(L_0<2\lambda_1\), as proposed in \cite{Liu26}.
\end{example}

\subsection{Regularity of SRDE}
\label{sec3.2}

In this part, we lift the regularity of the solution to Eq.~\eqref{spde} from \(\dot H^1\) to \(\dot H^\gamma\) for any \(\gamma \in (0, 1)\) and certain temporal H\"older regularity.  
We note that, for the Lipschitz diffusion case, such Sobolev and H\"older regularities were shown in \cite[Proposition 3.1, Theorem 3.3 and Corollary 3.2]{LQ21}, \cite[Lemma 4.2]{LS26}, \cite[Proposition 2.1]{Liu26}, and \cite[Lemma 3.1]{CL26}.

We will need the following well-known ultracontractivity and smoothing properties of the analytic \(\mathcal C_0\)-semigroup \(S\):
\begin{align}
\|S(t) u\|_\mu \le C e^{-ct} t^{-\frac{\mu-\nu}2} \|u\|_\nu,   
& \quad \forall~ t > 0, ~ 0 \le \nu \le \mu \le 2,~ u \in \dot H^\nu, \label{ana} \\
\|(S(t)-{\rm Id}_H) u\| \le C t^{\frac\rho2} \|u\|_\rho, \label{ana-hol} 
& \quad \forall~ t > 0, ~ 0 \le \rho \le 2, ~ u \in \dot H^\rho.
\end{align}

\begin{theorem} \label{tm-reg+}
Let $p\ge2$, $X_0 \in L_\omega^{p(q+1)} \dot H^1$, and Assumption \ref{ap} hold.
There exists a positive constant $C$ s.t. 
\begin{align}   \label{spde-hol}
\|X\|_{L_t^\infty L_\omega^p \dot H^1}
+ \|X\|_{\mathcal C_t^{1/2} L_\omega^p L_\xi^2}
& \le C e^{C T} (1+\|X_0\|^{q+1}_{L_\omega^{p(q+1)} \dot H^1}). 
\end{align}  
Assume furthermore that $X_0 \in L_\omega^p \dot H^{1+\gamma}$ with $\gamma \in (0, 1)$, then $X \in L_t^\infty L_\omega^p \dot H^{1+\gamma}$ s.t. 
\begin{align}  
\|X\|_{\mathcal C_t^{1/2} L_\omega^p \dot H^\gamma} 
+ \|X\|_{L_t^\infty L_\omega^p \dot H^{1+\gamma}}
& \le C e^{C T} (1+\|X_0\|_{L_\omega^{p (q+1)} \dot H^1}^{q+1}
+ \|X_0\|_{L_\omega^p \dot H^{1+\gamma}}), \quad d=1, \label{spde-1} \\
 \|X\|_{L_t^\infty L_\omega^p \dot H^{1+\gamma}}  
& \le C e^{C T} (1+\|X_0\|_{L_\omega^{p (q+1)} \dot H^1}^{q+1}
+ \|X_0\|_{L_\omega^p \dot H^{1+\gamma}}), \quad d=2,  \label{spde-2} \\ 
\|X\|_{L_t^\infty L_\omega^p \dot H^{1+\gamma}}  
& \le C e^{C T} (1+\|X_0\|_{L_\omega^{p_3} \dot H^1}^{q+1}
+ \|X_0\|_{L_\omega^p \dot H^{1+\gamma}}), \quad d=3,   \label{spde-3}
\end{align}   
provided $X_0 \in L_\omega^{p_3} \dot H^1 \cap L_\omega^p \dot H^{1+\gamma}$ (in 3-D case), where $p_3:=p (q+1)$ when $p \le 14/(q+1)$ and  
$p_3:=2p (q+8)/(16-pq)$ when $14/(q+1)<p < 16/q$. 
Moreover, under the condition \eqref{con-uni}, the estimates \eqref{spde-hol} and \eqref{spde-1} are uniform in time.

\end{theorem}

\begin{proof} 
By Young inequality, the conditions \eqref{f-grow} and \eqref{g-grow}, we get
\begin{align*}
& \int_0^T \mathbb E\|X\| \,\mathrm{d}t
+ \int_0^T \mathbb E \|F(X)\| \,\mathrm{d}t 
+\int_0^T \mathbb E \|G(X)\|^2_{\mathcal L_2^0} \,\mathrm{d}t \\  
& \le \|X\|_{L_{t, \omega}^1 L_\xi^2}
+ C (T + \|X\|^{q+1}_{L_{t, \omega}^{q+1} L_\xi^{2(q+1)}} )
+ C (T + \|X\|^2_{L_{t, \omega}^2 \dot H^2} + \|u\|_{L_\xi^{q+2}}^{q+2}),
\end{align*}
which is finite by the regularity $X \in L_\omega^p \mathcal C_t H \cap L_{t, \omega}^2 \dot H^1 \cap L_{t, \omega, \xi}^{q+2}$ with $p \ge 2(q+1)$ and the embedding \eqref{emb}.
By \cite[Remark G.0.6 and Proposition G.0.5(i)]{LR15} and the proof of \cite[Lemma 2.3]{LQ21}, the variational solution of Eq. \eqref{spde} is also a mild solution with $S(t):=e^{\Delta t}$ for all $t \in [0, T]$: 
\begin{align} \label{mild}
X(t) &=S(t) X_0+\int_0^t S(t-r) F(X(r)) \,\mathrm{d}r
+\int_0^t S(t-r) G(X(r)) \,\mathrm{d}W(r).
\end{align}    

Let $\gamma \in (0, 1)$.
For the first two terms in \eqref{mild}, by the regularity \eqref{est-spde}, we have   
\begin{align}
& \|S(t) X_0\|_{L_t^\infty L_\omega^p \dot H^{1+\gamma}}
+ \Big\| \int_0^t S(t-r) F(X(r)) \,\mathrm{d}r \Big\|_{L_t^\infty L_\omega^p \dot H^{1+\gamma}} \nonumber  \\ 
& \le C \|X_0\|_{L_\omega^p \dot H^{1+\gamma}} 
+ C \|F(X)\|_{L_t^\infty L_\omega^p L_\xi^2}
 \Big( \int_0^\infty r^{-\frac{1+\gamma}2} e^{- cr}   \,\mathrm{d}r \Big) \nonumber \\
 & \le C (1 + \|X_0\|_{L_\omega^p \dot H^{1+\gamma}} + \|X\|^{q+1}_{L_t^\infty L_\omega^{p (q+1)}\dot H^1}) \nonumber \\
 & \le C e^{CT} (1+ \|X_0\|_{L_\omega^p \dot H^{1+\gamma}} 
 + \|X_0\|^{q+1}_{L_\omega^{p (q+1)}\dot H^1} ). \label{F1+}
\end{align}    

By \eqref{ana} and BDG inequality, we have 
\begin{align} \label{G0}
& \Big\| \int_0^t S(t-r) G(X(r)) \,\mathrm{d}W(r) \Big\|^2_{L_t^\infty L_\omega^p \dot H^{1+\gamma}} \nonumber  \\
& \le C \sup_{0 \le t \le T} \int_0^t (t-r)^{-\gamma} e^{-2c (t-r)} \|G(X(r))\|_{L_\omega^p \mathcal L_2^1}^2  \,\mathrm{d}r.
\end{align} 
From the proof of \eqref{proof-G-grow} and the embeding \eqref{emb}, for any $v \in \dot H^1 \cap L_\xi^\infty$,  
 \begin{align} \label{G21}
\|G(v)\|^2_{\mathcal L_2^1}
& \le C(\|g'(v) \nabla v \|^2 + \|g(v)\|^2)
\le C (1 + \|v\|_1^{q/2+2} + \|v\|_{L_\xi^\infty}^{q/2} \|v \|_1^2).
\end{align}
When $d=1$, as $\|\cdot\|_{L_\xi^\infty} \le C \|\cdot\|_1$, we have 
$\|G(\cdot)\|^2_{\mathcal L_2^1}
\le C (1 + \|\cdot\|_1^{q/2+2})$ s.t.
\begin{align} \label{G1+1}
& \Big\| \int_0^t S(t-r) G(X(r)) \,\mathrm{d}W(r) \Big\|_{L_t^\infty L_\omega^p \dot H^{1+\gamma}} \nonumber  \\ 
& \le C \|G(X)\|_{L_t^\infty L_\omega^p \mathcal L_2^1} ( \int_0^\infty r^{-\gamma} e^{-2c r} \,\mathrm{d}r )^{1/2}\nonumber \\
& \le C (1 + \|X\|^{q/4+1}_{L_t^\infty L_\omega^{p(q/4+1)} \dot H^1})
\le e^{CT} (1 + \|X\|^{q/4+1}_{L_\omega^{p(q/4+1)} \dot H^1}).
\end{align}   
From \eqref{G21} and the estimate $\|\cdot\|_{L_\xi^\infty} \le C \|\cdot\|_2^\beta \|\cdot\|_1^{1-\beta}$ for $\beta=1/2$ when $d=3$ and any $\beta \in (0, 1)$ when $d=2$, it is clear that  
 \begin{align} \label{G2}
\|G(v)\|^2_{\mathcal L_2^1}
& \le C (1 + \|v\|_1^{q/2+2} + \|v\|_2^{\beta q/2} \|v\|_1^{(1-\beta) q/2+2}), 
 \quad v \in \dot H^2.
\end{align}   
By \eqref{G0} and H\"older (or Young) inequality, for any $\alpha_i >1$ with $\alpha_i':=\alpha_i/(\alpha_i-1)$, $i=1,2,3$, s.t. $\alpha_1' \gamma<1$,
\begin{align} \label{G2+}
& \Big\| \int_0^t S(t-r) G(X(r)) \,\mathrm{d}W(r) \Big\|_{L_t^\infty L_\omega^p \dot H^{1+\gamma}}\\
% & \le C \sup_{0 \le t \le T} ( \int_0^t r^{-\alpha \gamma} e^{-c r} \,\mathrm{d}r)^{1/(2 \alpha)} (\int_0^t \|G(X(r))\|_{L_\omega^p \mathcal L_2^1}^{2 \alpha'}  \,\mathrm{d}r )^{1/(2 \alpha')} \nonumber \\
& \le C \Big( \int_0^\infty r^{-\alpha_1' \gamma} e^{-2c \alpha_1' r} \,\mathrm{d}r\Big)^{1/(2 \alpha_1')} 
\|G(X)\|_{L_t^{2 \alpha_1} L_\omega^p \mathcal L_2^1} \nonumber   \\
& \le C \|1 + \|X\|_2^{\beta q/4} \|v\|_1^{[(1-\beta) q+4]/4}
 + \|X\|_1^{q/4+1}\|_{L_t^{2 \alpha_1} L_\omega^p} \nonumber  \\
% & \le C (1 +  \|v\|^{q/4+1}_{L_t^{\alpha_1 (q+4)/2} L_\omega^{p q/4+1} \dot H^1}\nonumber  \\
% & \qquad + \|v\|^{\beta q/4}_{L_t^{2 \alpha_1 \alpha_3 \beta q/4} L_\omega^{p \alpha_2 \beta q/4} \dot H^2} \|v\|_{L_t^{2 \alpha_1 \alpha_3' [(1-\beta) q+4]/4} L_\omega^{p \alpha_2' [(1-\beta) q+4]/4} \dot H^1}^{[(1-\beta) q+4]/4} )\nonumber  \\
& \le C (1 +  \|X\|^{q/4+1}_{L_t^\infty L_\omega^{p q/4+1} \dot H^1}
+ \|X\|^{\beta q/4}_{L_t^{2 \alpha_1 \alpha_3 \beta q/4} L_\omega^{p \alpha_2 \beta q/4} \dot H^2} \|X\|_{L_t^\infty L_\omega^{p \alpha_2' [(1-\beta) q+4]/4} \dot H^1}^{[(1-\beta) q+4]/4} ).  \nonumber
\end{align}    
Now we choose $2 \alpha_1 \alpha_3 \beta q/4 \le 2$ and $p \alpha_2 \beta q/4 \le 2$ s.t.
\begin{align*}
& \Big\| \int_0^t S(t-r) G(X(r)) \,\mathrm{d}W(r) \Big\|_{L_t^\infty L_\omega^p \dot H^{1+\gamma}} \nonumber  \\
& \le C (1 +  \|X\|^{q/4+1}_{L_t^\infty L_\omega^{p q/4+1} \dot H^1}
+ \|X\|^{\beta q/4}_{L_{t, \omega}^2 \dot H^2} \|X\|_{L_t^\infty L_\omega^{p \alpha_2' [(1-\beta) q+4]/4} \dot H^1}^{[(1-\beta) q+4]/4} ).  
\end{align*}   

To ensure $\alpha_1' \gamma<1$ and $2 \alpha_1 \alpha_3 \beta q/4=p \alpha_2 \beta q/4=2$ when $d=2$, we choose $\alpha_1'=\gamma^{-1}-$ s.t. $\alpha_1 = (1-\gamma)^{-1}+$, $\alpha_3=4(1-\gamma)/(\beta q)-$, and $\alpha_2=8/(\beta p q)$ s.t. $\alpha_2'=8/(8-\beta p q)$ with $\beta p q<8$.
We take $\beta \in (0, \max\{6[p(q+1)-2]^{-1}, 1\})$ s.t. $p (q+1) \ge 2 p[(1-\beta) q+4]/(8-\beta pq)$ and $\beta p q<8$, and thus   
\begin{align} \label{G1+2}
 \Big\| \int_0^t S(t-r) G(X(r)) \,\mathrm{d}W(r) \Big\|_{L_t^\infty L_\omega^p \dot H^{1+\gamma}}  
& \le C e^{C T} (1 + \|X_0\|_{L_\omega^{p(q+1)} \dot H^1}^{q/4+1}).  
\end{align}    

When $d=3$, we take $\beta=1/2$ and choose $\alpha_1'=\gamma^{-1}-$ s.t. $\alpha_1 = (1-\gamma)^{-1}+$, $\alpha_3=8(1-\gamma)/q -$, and $\alpha_2=16/(pq)$ s.t. $\alpha_2'=16/(16-pq)$ with $p q<16$.
If $p \le 14/(q+1)$, then $p (q+1) \ge 2 p[(1-\beta) q+4]/(8-\beta pq)$; if 
$14/(q+1)<p < 16/q$, then $p (q+1) < 2 p[(1-\beta) q+4]/(8-\beta pq)$.
Consequently, 
\begin{align} \label{G1+3}
&\Big \| \int_0^t S(t-r) G(X(r)) \,\mathrm{d}W(r) \Big\|_{L_t^\infty L_\omega^p \dot H^{1+\gamma}}  \nonumber \\
& \le 
\begin{cases}
C e^{C T} (1 + \|X_0\|_{L_\omega^{p(q+1)} \dot H^1}^{q/4+1}), & \quad p \le 14/(q+1), \\
C e^{C T} (1 + \|X_0\|_{L_\omega^{2p (q+8)/(16-pq)} \dot H^1}^{q/4+1}), & \quad 14/(q+1)<p < 16/q.
\end{cases}
\end{align}   
Then we conclude the Sobolev estimates in \eqref{spde-1} (for the second term), \eqref{spde-2}, and \eqref{spde-3} from \eqref{F1+}, \eqref{G1+1}, \eqref{G1+2}, and \eqref{G1+3}.
 
To derive the temporal H\"older regularity, without loss of generality, let $0\le s<t\le T$.
By Minkovskii inequality, we have for each $\gamma \in [0, 1)$, 
\begin{align*}
\|X(t)-X(s)\|_{L_\omega^p \dot H^\gamma} 
&\le \Big\|[S(t)-{\rm Id}] X(s)\Big\|_{L_\omega^p \dot H^\gamma}
+ \Big\|\int_s^t S(t-r) F(X(r)) \,\mathrm{d}r \Big\|_{L_\omega^p \dot H^\gamma} \\& 
\quad + \Big\|\int_s^t S(t-r) G(X(r)) \,\mathrm{d}W(r) \Big\|_{L_\omega^p \dot H^\gamma}
=: I+II+III.
\end{align*} 
The smoothing property \eqref{ana-hol} implies that 
\begin{align} \label{I}
I & \le C \|X\|_{L_t^\infty L_\omega^p \dot H^{1+\gamma}}
 (t-s)^{1/2}.
\end{align} 
For the second term $II$, by the ultracontractive property \eqref{ana}, the embedding \eqref{emb}, and the estimate \eqref{est-spde} we have
\begin{align} \label{II}
II &\le C \|F(X)\|_{L_t^\infty L_\omega^p L_x^2} (t-s)^{1-\gamma/2}
\le C (1+\|X\|^{q+1}_{L_t^\infty L_\omega^{p(q+1)} \dot H^1}) (t-s)^{1-\gamma/2}.
\end{align}

For the last term $III$, let us first consider the 1-D case. 
By BDG inequality, the smoothing property \eqref{ana}, the condition \eqref{G21}, and the estimate \eqref{est-spde}, we get 
\begin{align} \label{III}
III \le C \|G(X)\|_{L_t^\infty L_\omega^p \mathcal L_2^1}
(t-s)^{1/2}
\le C  (1+\|X\|^{q/4+1}_{L_t^\infty L_\omega^{p(q/4+1)} \dot H^1}) (t-s)^{1/2}.
\end{align}
Similarly, for $\gamma=0$, then the condition \eqref{g-grow} yields 
\begin{align} \label{III-}
III \le C \|G(X)\|_{L_t^\infty L_\omega^p \mathcal L_2^0}
(t-s)^{1/2}
\le C  (1+\|X\|^{q/4+1}_{L_t^\infty L_\omega^{p(q/4+1)} \dot H^1}) (t-s)^{1/2}.
\end{align} 
Then we conclude \eqref{spde-hol} and \eqref{spde-1} by combining the estimates \eqref{I}, \eqref{II}, \eqref{III}, \eqref{III-}, and \eqref{est-spde}.
\end{proof}

\begin{remark} 
In the 2-D and 3-D cases, using BDG inequality, condition \eqref{G2}, and estimates \eqref{G1+1}, \eqref{G1+2}-\eqref{G1+3}, we have 
$III  \le \|G(X)\|_{L_t^\infty L_\omega^p \mathcal L_2^1} (t-s)^{1/2}$.
However, according to \eqref{G2}, the term $\|G(X)\|_{L_t^\infty L_\omega^p \mathcal L_2^1}$ is not expected to be finite.
\end{remark}

\section{Strong convergence rates of tamed-FEM}
\label{sec4}

In this section, we construct a family of tamed-FEMs for Eq. \eqref{spde} with superlinear coefficients, and establish both their long-time unconditional stability and their optimal strong convergence rates.

\subsection{Tamed-FEM and Long-time Stability}
\label{sec4.1}

Let \(h\in(0,1)\), \(\mathcal T_h\) be a regular family of quasi-uniform partitions of \(\mathcal O\) with maximal mesh size \(h\), and \(V_h\subset \dot H^1\) be the space of continuous functions on \(\bar{\mathcal O}\) which are piecewise linear over \(\mathcal T_h\) and vanish on \(\partial\mathcal O\). Let $A_h: V_h \rightarrow V_h$ and $\mathcal P_h: \dot H^{-1} \rightarrow V_h$ be the discrete Dirichlet Laplacian and the generalized orthogonal projection operators, resp., defined by
\begin{align*}  
\langle A_h u^h, v^h\rangle & =-\langle \nabla u^h, \nabla v^h\rangle,
\quad u^h, v^h\in V_h,  \\
\langle \mathcal P_h u, v^h\rangle & = \langle u, v^h\rangle_{V^*, V},
\quad u\in \dot H^{-1},\ v^h\in V_h. 
\end{align*}

Denote by \(\tau\in(0,1)\) the temporal step-size.
The tamed-FEM is to find a \(V_h\)-valued (time-homogeneous) Markov chain $(Y_n^h, \mathcal F_{t_n})$ s.t. 
 \begin{align} \label{t-fem}  
  Y_n^h & =Y_{n-1}^h + A_h Y_n^h \tau
  + \mathcal P_h F_\tau (Y_{n-1}^h) \tau 
  + \mathcal P_h G_\tau(Y_{n-1}^h)\delta_{n-1} W, 
  \quad n \in \mathbb N_+,
\end{align} 
for some tamed operators $F_\tau$ and $G_\tau$ (corresponding to certain tamed functions $f_\tau$ and $g_\tau$, resp.), where $\delta_{n-1} W:=W(t_n) -W(t_{n-1})$, $t_n=n \tau$, $0 \le n \le N$, with initial datum $Y^h_0=\mathcal P_h X_0$.  
The tamed-FEM \eqref{t-fem} is linearly implicit, so it can be uniquely solved pathwise. 
It can be equivalently rewritten as  
\begin{align}\label{full} 
Y_n^h & = S_{h,\tau} Y_{n-1}^h+\tau S_{h,\tau} \mathcal P_h F_\tau(Y_{n-1}^h)
+ S_{h,\tau} \mathcal P_h G_\tau(Y_{n-1}^h) \delta_{n-1} W, 
\quad n \in \mathbb N_+, 
\end{align}
where $S_{h,\tau}:=({\rm Id}-\tau A_h)^{-1}$, with ${\rm Id}$ denoting the identity operator in $V_h$, is an approximation of $S$ in one step.
Iterating \eqref{full} yields
\begin{align} \label{full-sum}
Y^h_n 
=S_{h,\tau}^n  Y^h_0+\tau \sum_{i=0}^{n-1}  S_{h,\tau}^{n-i} \mathcal P_h F_\tau(Y^h_i)
+\sum_{i=0}^{n-1}  S_{h,\tau}^{n-i} \mathcal P_h G_\tau(Y^h_i) \delta_i W,
\quad n \in \mathbb N_+.
\end{align}

We begin with the following Lyapunov estimate for general \(p\ge 2\), which is stronger than the one given in \cite{LS26} for \(p=2\) in the Lipschitz diffusion case.

\begin{theorem} \label{tm-lya}
Let $p \ge 2$, $X_0 \in L_\omega^p H$, Assumption \ref{ap} hold, and assume that there exist constants $\tau_0 \in (0, 1]$, $N_0 \in \mathbb R$, and $N_0^*>0$ s.t. for all $\tau \in (0, \tau_0)$ and $u \in \dot H^1$,
\begin{align} \label{coe-tau}
2 \langle F_\tau(u), u \rangle_{V^*, V}
+ \tau \|F_\tau(u)\|^2
 + (p-1) \|G_\tau (u)\|_{\mathcal L_2^0}^2
 \le N_0 - N_0^* \|u\|^2.
\end{align}
Then there exist positive constants $T_1$, $T_2$, and $\tau_1 \in (0, 1]$ s.t. for any $\tau \in (0, \tau_1)$, 
\begin{align} \label{lya}
E_{n-1} \|Y^h_n\|^p  
\le T_1 \tau + (1 - T_2 \tau) \|Y_{n-1}^h\|^p, \quad n \in \mathbb N_+.
 	\end{align}    
\end{theorem}

 \begin{proof}
 Let $n \in \mathbb N_+$. 
Testing with $Y_n^h$ on both sides of the tamed-FEM \eqref{t-fem} and using the elementary equality
\begin{align}\label{ab}
	2\langle a-b,a\rangle =|a|^2-|b|^2+|a-b|^2, 
\end{align}
for any $a,b$ in an arbitrary Hilbert space,
and Cauchy--Schwarz inequality, we obtain  
\begin{align*}
& \|Y^h_n\|^2-\|Y^h_{n-1}\|^2 
+  \|Y^h_n-Y^h_{n-1}\|^2 +2 \tau \|\nabla Y^h_n\|^2 \nonumber  \\
&=2 \langle Y^h_n-Y^h_{n-1}, 
\tau F_\tau(Y^h_{n-1}) + G_\tau (Y^h_{n-1}) \delta_{n-1} W\rangle
 \nonumber  \\
& \quad + 2 \langle Y^h_{n-1}, \tau F_\tau(Y^h_{n-1}) + G_\tau (Y^h_{n-1}) \delta_{n-1} W\rangle \nonumber \\
& \le \|\tau \mathcal P_h F_\tau(Y^h_{n-1}) + \mathcal P_h G_\tau (Y^h_{n-1}) \delta_{n-1} W\|^2 + \|Y^h_n-Y^h_{n-1}\|^2 \nonumber \\
& \quad + 2 \langle Y^h_{n-1}, \tau F_\tau(Y^h_{n-1}) + G_\tau (Y^h_{n-1}) \delta_{n-1} W\rangle.
\end{align*}  
This shows 
\begin{align} \label{lya-err}
& \|Y^h_n\|^2+2 \tau \|\nabla Y^h_n\|^2
\le \|Y^h_{n-1} + \tau \mathcal P_h F_\tau(Y^h_{n-1}) + \mathcal P_h G_\tau (Y^h_{n-1}) \delta_{n-1} W \|^2.
\end{align}   
Denote by $T_{n-1}:=Y^h_{n-1} + \tau \mathcal P_h F_\tau(Y^h_{n-1}) + \mathcal P_h G_\tau (Y^h_{n-1}) \delta_{n-1} W$ and set
$I_{n-1}(t):=Y_{n-1}^h + \tau \mathcal P_h F_\tau(Y^h_{n-1}) 
+ \mathcal P_h G_\tau (Y^h_{n-1}) [W(t_{n-1}+t)-W(t_{n-1})]$ for $0\le t\le \tau$.
Then $I_{n-1}(0)=Y^h_{n-1} + \tau \mathcal P_h F_\tau(Y^h_{n-1}) $ and $I_{n-1}(\tau)=R_{n-1}$.  
Applying the It\^o formula \eqref{ito} to $\|I_{n-1}(t)\|^p$ (and using a stopping time argument if necessary), followed by taking $\mathbb E_{n-1}$, and using the estimate  
$\|[\mathcal P_h G_\tau (Y^h_{n-1})]^* I_{n-1}(t)\|_{U_0}
 \le \|G_\tau (Y^h_{n-1})\|_{\mathcal L_2^0} \|I_{n-1}(t)\|$,
and H\"older inequality, we infer
\begin{align*}
 \frac{\,\mathrm{d}}{\,\mathrm{d} t}\mathbb E_{n-1}\|I_{n-1}(t)\|^p 
&=\frac p2\mathbb E_{n-1} [\|I_{n-1}(t)\|^{p-2}\|\mathcal P_h G_\tau (Y^h_{n-1})\|_{\mathcal L_2^0}^2 ] \\
& \quad + \frac12 p(p-2) \mathbb E_{n-1} [\|I_{n-1}(t)\|^{p-4}\|[\mathcal P_h G_\tau (Y^h_{n-1})]^*I_{n-1}(t)\|_{U_0}^2 ]\\
&\le \frac12 p(p-1) \|G_\tau (Y^h_{n-1})\|_{\mathcal L_2^0}^2 \mathbb E_{n-1}\|I_{n-1}(t)\|^{p-2}.
\end{align*} 
By H\"older inequality, we have
$\mathbb E_{n-1}\|I_{n-1}(t)\|^{p-2} \le (\mathbb E_{n-1}\|I_{n-1}(t)\|^p)^{1-2/p}$.
Substituting this into the above estimate yields
$$\frac{\,\mathrm{d}}{\,\mathrm{d} t} (\mathbb E_{n-1} \|I_{n-1}(t)\|^p)^{2/p} \le (p-1)\|G_\tau (Y^h_{n-1})\|_{\mathcal L_2^0}^2.$$
Integrating over $[0,\tau]$ and using the condition \eqref{coe-tau} yield
\begin{align*}
\mathbb E_{n-1} \|T_{n-1}\|^p 
& \le (\|Y^h_{n-1} + \tau \mathcal P_h F_\tau(Y^h_{n-1})\|^2 
+ (p-1) \tau \|G_\tau (Y^h_{n-1})\|_{\mathcal L_2^0}^2 )^{p/2} \nonumber \\
& \le (\|Y_{n-1}^h\|^2 + \tau [ 2 \langle F_\tau(Y^h_{n-1}), Y_{n-1}^h \rangle_{V^*, V} \\
& \quad + \tau \|F_\tau(Y^h_{n-1})\|^2
 + (p-1) \|G_\tau (Y^h_{n-1})\|_{\mathcal L_2^0}^2]  )^{p/2} \\
 & \le [(1 - N_0^* \tau)\|Y_{n-1}^h\|^2 + N_0 \tau)^{p/2}.
\end{align*}
Using the elementary inequality  
\begin{align} \label{in-ep}
(a+ \tau b)^\alpha \le (1+\epsilon \tau) a^\alpha + C_{p, \epsilon} \tau b^\alpha,
\end{align} 
for any $a, b \ge 0$, $\alpha \ge 1$, and $\epsilon \in (0, 1)$ and some $C_{\alpha, \epsilon}>0$, we obtain 
\begin{align*}
\mathbb E_{n-1} \|T_{n-1}\|^p 
& \le (1+\epsilon \tau) (1 - N_0^* \tau)^{p/2} \|Y_{n-1}^h\|^p + C_{p, \epsilon} |N_0|^{p/2} \tau.
\end{align*}
This estimate, in combination with \eqref{lya-err}, yields 
\begin{align*}
E_{n-1} \|Y^h_n\|^p 
& \le E_{n-1} [\|Y^h_n\|^2+2 \tau \|\nabla Y^h_n\|^2]^{p/2} \\
& \le (1+\epsilon \tau) (1 - N_0^* \tau)^{p/2} \|Y_{n-1}^h\|^p + C_{p, \epsilon} |N_0|^{p/2} \tau.
\end{align*} 
This shows \eqref{lya} with $T_1=C_{p, \epsilon} |N_0|^{p/2}$ and certain $T_2>0$.
\end{proof}

Next, let us show that the tamed functions constructed in \cite{LS26} satisfy the coercivity condition \eqref{coe-tau} in the superlinear diffusion case.

\begin{lemma} \label{lm-ex-tau}
Let $f$ and $g$ satisfy Assumption \ref{ap}.
Then  
\begin{align} \label{fg-tau}
 f_\tau (\xi):=\frac{f(\xi)}{(1+\tau|\xi|^{2q})^{1/2}}, \quad 
 g_\tau(\xi):=\frac{g(\xi)}{(1+\sqrt{\tau}|\xi|^q)^{1/2}}, \quad 
\xi \in \mathbb R,
\end{align} 
satisfy the tamed coercivity condition \eqref{coe-tau}.
\end{lemma}

\begin{proof} 
With \(s:=|\xi|\),
\[
2\xi f_\tau(\xi)+\tau |f_\tau(\xi)|^2
\le
\frac{2(A_1+A_2s^2-A_3s^{q+2})}{\sqrt{1+\tau s^{2q}}}
+
\frac{\tau A_4^2(1+s^{q+1})^2}{1+\tau s^{2q}}.
\]
For \(s\) large, the first term behaves like
$ - 2A_3 \tau^{-1/2} s^2$, 
while the second term is bounded by \(C(1+s^2)\). Hence, there exist positive constants \(M_0, M_1\), and \(\tau_1 \in(0,1]\) s.t. 
\begin{equation}\label{eq:tamed-drift}
2\xi f_\tau(\xi)+\tau |f_\tau(\xi)|^2
\le M_0 - M_1 \tau^{-1/2} |\xi|^2, \quad \xi\in\mathbb R, ~\tau \in (0, \tau_1).
\end{equation}
Indeed, for \(|\xi|\le R\) the left-hand side is uniformly bounded in \(\tau\), and for \(|\xi|\ge R\) the negative term dominates by choosing \(\tau_1\) sufficiently small.

For large \(s=|\xi|\), 
\[
|g_\tau(\xi)|^2
=
\frac{|g(\xi)|^2}{1+\sqrt{\tau}|\xi|^q}
\le
C\,\frac{1+ s^{2+q/2}}{1+\sqrt{\tau} s^q}, 
\]
behaves like
$C\,\tau^{-1/2} s^{2-q/2}$. 
Since the exponent \(2-q/2<2\), as \(q>0\), for every \(\epsilon>0\) there exists \(C_\epsilon>0\) independent of \(\tau\) s.t.
\begin{equation}\label{eq:tamed-diffusion}
(p-1)C_Q |g_\tau(\xi)|^2
\le
C_\epsilon
+
\epsilon M_1  \tau^{-1/2} |\xi|^2.
\end{equation} 
 
Adding \eqref{eq:tamed-drift} and \eqref{eq:tamed-diffusion}, and choosing \(\epsilon\le 1/2\), we obtain \eqref{coe-tau} with \(N_0=M_0+C_\epsilon\) and $N_0^* := \frac{M_1}{2\sqrt{\tau_0}}>0$. 
\end{proof}

\subsection{Assumptions on Tamed-FEM}
\label{sec4.2}

To derive strong error estimates for the tamed-FEM \eqref{t-fem} applied to Eq.~\eqref{see}, we impose the following conditions on the tamed drift and diffusion functions, and then provide concrete examples that satisfy both these conditions and Assumption~\ref{ap}.

\begin{assumption}\label{ap-tau} 
There exist a positive constant $C$ s.t. for any $\xi, \eta \in \mathbb R$,  
\begin{align}
|f_\tau(\xi)| + |g_\tau(\xi)|
& \le C (1+ |\xi|^{q+1}), \label{fgtau-grow} \\ 
|f(\xi)-f_\tau(\xi)| + |g(\xi)-g_\tau(\xi)|
& \le C \tau^{1/2} (1+ |\xi|^{2q+1}), \label{f-ftau} \\
|f_\tau(\xi)-f_\tau(\eta)| + |g_\tau(\xi)-g_\tau(\eta)| 
& \le C(1+ |\xi|^q + |\eta|^q) |\xi-\eta|. \label{ftau-}  
\end{align}  
\end{assumption}

\begin{assumption}\label{ap-tau'} 
For any $p^*>2$, there exist constants $\tau_2 \in (0, 1]$ and $M_0 \in \mathbb R$ s.t.  
\begin{align}  \label{fgtau'}
2 f_\tau'(\xi) + \tau |f_\tau'(\xi)|^2 + C_Q (p^*&  -1) |g_\tau'(\xi)|^2 
 \le M_0, \quad \xi \in \mathbb R, ~\tau \in (0, \tau_2). 
\end{align}  
\end{assumption}

\begin{remark}  \label{rk-fgtau'}
In the linear growth diffusion case (without taming the diffusion function), \eqref{fgtau'} is ensured by the following uniform upper bound condition for \(f_\tau'\), namely
$f_\tau'\le M_1$ for some \(M_1\in\mathbb R\), first proposed in \cite{LS26}:
\begin{align} \label{ftau'} 
[1+\tau |\xi|^{2q}]f'(\xi) - q \tau |\xi|^{2(q-1)} \xi f(\xi) 
\le M_1 (1+\tau|\xi|^{2q})^{3/2}, \quad \xi \in \mathbb R. 
\end{align} 
This condition encompasses all odd-degree polynomials with negative leading coefficients.
In addition, for \eqref{ftau'} and \(g\) with Lipschitz constant \(L_g\), \cite[Lemma 3.2]{CL26} proved that $2 f_\tau'(\xi) + \tau |f_\tau'(\xi)|^2 \le 2M_1$ s.t. \eqref{fgtau'} holds with $M_0=2M_1+C_Q (p^*-1) L_g^2$.
\end{remark}

\begin{lemma} \label{lm-ftau} 
Under Assumption \ref{ap-tau}, there exists a positive constant $C$ s.t. 
\begin{align}   
\|F(u)-F_\tau(u)\|_{-1} 
& \le C \tau^{1/2} (1+\|u\|_1^{2q+1}),   \label{F-Ftau}  \\ 
\|F_\tau(u)-F_\tau(v)\|_{-1}
& \le C (1+\|u\|^q_1+\|v\|^q_1 ) \|u-v\|,  \label{Ftau-} \\  
\|G(u)-G_\tau(u)\|_{\mathcal L_2^{-1}}
& \le C \tau^{1/2} ( 1+ \|u\|_1^{2q+1}), \label{G-Gtau} \\ 
\|G_\tau(z)-G_\tau(w)\|_{\mathcal L_2^0} 
& \le C (1+\|z\|_{L_\xi^\infty}^q+\|w\|_{L_\xi^\infty}^q) \|z-w\|^2, \label{Gtau-} 
\end{align}    
for any $u, v \in \dot H^1$ and $z, w \in L_\xi^\infty$.
Under Assumption \ref{ap-tau'},  
\begin{align}     \label{FGtau-}  
&  2 \langle F_\tau(u)-F_\tau(v), u-v \rangle_{V^*, V} + \tau \|F_\tau(u)-F_\tau(v)\|^2 \nonumber \\
&  \quad + (p^*-1) \|G_\tau(u)) -G_\tau(v)\|_{\mathcal L_2^0}^2 
 \le M_0 \|u-v\|^2.   
\end{align}    
\end{lemma}

\begin{proof} 
By the condition \eqref{f-ftau} and the embedding \eqref{emb}, we get \eqref{F-Ftau}:
\begin{align*} 
\|F(u)-F_\tau(u)\|_{-1}
& \le C \tau^{1/2} \|1+ |u|^{2q+1}\|_{L_\xi^{2(q+1)/(2q+1)}}
\le C \tau^{1/2} ( 1+ \|u\|_1^{2q+1}),
\end{align*} 
Similarly, the condition \eqref{ftau-} and the embedding \eqref{emb} yield  \eqref{Ftau-}: 
\begin{align*} 
\|F_\tau(u)-F_\tau(v)\|_{-1}
& \le C \|F_\tau(u)-F_\tau(v)\|_{L_\xi^{2(q+1)/(2q+1)}}
\le C (1+\|u\|_1^q+\|u\|_1^q)\|u-v\|.
\end{align*}  
The estimates \eqref{G-Gtau}-\eqref{Gtau-} can be derived from conditions \eqref{f-ftau}-\eqref{ftau-}, resp., together with \eqref{emb}; the calculations are analogous, and we omit the details.
  
Finally, to show \eqref{FGtau-} for any $u, v \in \dot H^1$, we observe that $f_\tau(u)-f_\tau(v)=M_\tau(u, v)(u-v)$ and $g_\tau(u)-g_\tau(v)=N_\tau(u, v)(u-v)$ with $M_\tau(u, v):=\int_0^1 f_\tau'(v+s(u-v)) \,\mathrm{d}s$ $N_\tau(u, v):=\int_0^1 g_\tau'(v+s(u-v)) \,\mathrm{d}s$.
Then
\begin{align*}
& 2 \langle u-v , F_\tau(u)-F_\tau(v) \rangle_{V^*, V} + \tau \|F_\tau(u)-F_\tau(v)\|^2 \nonumber + (p^*-1) \|G_\tau(u))-G_\tau(v)\|_{\mathcal L_2^0}^2 \\
& = \int_{\mathcal{O}} [ 2 (f_\tau(u) - f_\tau(v))(u-v) + \tau |f_\tau(u) - f_\tau(v)|^2 \\
& \quad + (p^*-1) \sum_{m \in \mathbb N_+} q_m |[g_\tau(u) - g_\tau(v)] g_m|^2 ] \, d\xi \\
& = \int_{\mathcal{O}} [ 2 (f_\tau(u) - f_\tau(v))(u-v) + \tau |f_\tau(u) - f_\tau(v)|^2
+ C_Q (p^*-1) |g_\tau(u) - g_\tau(v)|^2 ] \, d\xi \\
& = \int_{\mathcal{O}} [M_\tau(u, v) + \tau |M_\tau(u, v)|^2 + C_Q (p^*-1) |N_\tau(u, v)|^2] |u-v|^2 \, d\xi.
\end{align*}
By Cauchy--Schwarz inequality, we obtain
\begin{align*}
& 2 \langle u-v , F_\tau(u)-F_\tau(v) \rangle_{V^*, V} + \tau \|F_\tau(u)-F_\tau(v)\|^2 \nonumber + (p^*-1) \|G_\tau(u))-G_\tau(v)\|_{\mathcal L_2^0}^2 \\
& \leq \int_{\mathcal{O}} H(u,v) |u - v|^2 \, d\xi,
\end{align*} 
where 
\begin{align*}
H(u,v):=& \int_0^1 [f_\tau'(v+s(u-v)) + \tau |f_\tau'(v+s(u-v))|^2 + \|g_\tau'(v+s(u-v))|^2] \, \,\mathrm{d}s \\
\le & \sup_{\xi \in \mathbb R} \{f_\tau'(\xi) + \tau |f_\tau'(\xi)|^2 + C_Q (p^*-1) \|g_\tau'(\xi)|^2\}.
\end{align*} 
In combination with the above inequality and \eqref{fgtau'}, we conclude \eqref{FGtau-}.
\end{proof}

\begin{lemma} \label{lm-ex-tau}
Let $f$ and $g$ satisfy Assumption \ref{ap}.
Then $f_\tau$ and $g_\tau$ given by 
\eqref{fg-tau} satisfy all the conditions in Assumption \ref{ap-tau}.
\end{lemma}

\begin{proof} 
Let $\xi, \eta \in \mathbb R$.
The first inequality \eqref{fgtau-grow} follows immediately from \eqref{f-grow}.
Define
$b(t):=1-(1+t)^{-1/2}$, $t \ge 0$.
Direct computations yield   
\begin{align*}
b(t) =\frac{t}{(1+t) [(1+t)^{1/2}+1]} 
% \le \min \{ \frac{t}{2 t^{1/2}}, \frac t2 \}
 \le \min \{ t^{1/2}, t/2 \},  \quad t \ge 0.
\end{align*}  
Then we use the representation \eqref{fg-tau} to derive 
\begin{align*}
|f(\xi)-f_\tau(\xi)|
& =b(\tau|\xi|^{2q}) |f(\xi)|
\le \tau^{1/2} |\xi|^q |f(\xi)|, \\
|g(\xi)-g_\tau(\xi)|
& = b(\sqrt{\tau}|\xi|^q) |g(\xi)|
\le \frac12 \tau^{1/2} |\xi|^q |g(\xi)|.
\end{align*}   
By \eqref{f-grow}, \eqref{g-grow}, and Young inequality, we derive \eqref{f-ftau}. 

To show \eqref{ftau-}, we use the representation \eqref{fg-tau} (see \cite[Lemma 2]{LS26}) to get 
\begin{align*}
& f_\tau(\xi)-f_\tau(\eta)  \\
% &=\frac{f(\xi)}{(1+\tau|\xi|^{2q})^{1/2}}-\frac{f(\eta)}{(1+\tau|\eta|^{2q})^{1/2}}  \\
% &= \frac{f(\xi) (1+\tau|\eta|^{2q})^{1/2}-f(\eta)(1+\tau|\xi|^{2q})^{1/2}}{(1+\tau|\xi|^{2q})^{1/2} (1+\tau|\eta|^{2q})^{1/2}} \nonumber  \\
% &\le \frac{|f(\xi)-f(\eta)| (1+\tau|\eta|^{2q})^{1/2}}{(1+\tau|\xi|^{2q})^{1/2} (1+\tau|\eta|^{2q})^{1/2}} \\
% & \quad + \frac{|f(\eta)| \cdot |(1+\tau|\eta|^{2q})^{1/2}-(1+\tau|\xi|^{2q})^{1/2}]|}{(1+\tau|\xi|^{2q})^{1/2} (1+\tau|\eta|^{2q})^{1/2}}   \\
&= \frac{f(\xi)-f(\eta)}{(1+\tau|\xi|^{2q})^{1/2}} 
+ \frac{f(\eta) [(1+\tau|\eta|^{2q})^{1/2}-(1+\tau|\xi|^{2q})^{1/2}]}{(1+\tau|\xi|^{2q})^{1/2} (1+\tau|\eta|^{2q})^{1/2}}  \nonumber   \\
&= \frac{f(\xi)-f(\eta)}{(1+\tau|\xi|^{2q})^{1/2}} 
+ \frac{\tau f(\eta) (|\eta|^{2q}-|\xi|^{2q})} 
{(1+\tau|\xi|^{2q})^{1/2} (1+\tau|\eta|^{2q})^{1/2}[(1+\tau|\eta|^{2q})^{1/2}+(1+\tau|\xi|^{2q})^{1/2}]}.
\end{align*} 
The inequality
$||\eta|^{2q}-|\xi|^{2q}| \le C |\xi-\eta| (1+|\eta|^{2q-1}+|\xi|^{2q-1})$ and \eqref{f-grow} imply 
\begin{align*}
|f_\tau(\xi)-f_\tau(\eta)|  
& \le |f(\xi)-f(\eta)|
+ \frac{C \tau (1+|\eta|^{q+1}) \cdot |\xi-\eta| (1+|\eta|^{2q-1}+|\xi|^{2q-1})} 
{1+\tau|\xi|^{2q}+\tau|\eta|^{2q}}  \\  
& \le |f(\xi)-f(\eta)|
+ C  |\xi-\eta| \cdot \frac{ \tau(1+ |\xi|^{3q}+ |\eta|^{3q})}{1+\tau|\xi|^{2q}+ \tau |\eta|^{2q}} \\ 
&\le C |\xi-\eta| (1+ |\eta|^q+ |\xi|^q).
\end{align*} 
As $g$ satisfies \eqref{g-grow}, similar argument implies 
\begin{align*}
|g_\tau(\xi)-g_\tau(\eta)|  
& \le |g(\xi)-g(\eta)|
+ \frac{C \sqrt{\tau} (1+|\eta|^{q/4+1}) \cdot |\xi-\eta| (1+|\eta|^{q-1}+|\xi|^{q-1})} 
{1+ \sqrt{\tau} |\xi|^q + \sqrt{\tau} |\eta|^q}  \\  
& \le |g(\xi)-g(\eta)|
+ C  |\xi-\eta| \cdot \frac{ \sqrt{\tau}(1+ |\xi|^{5q/4}+ |\eta|^{5q/4})}{1+\sqrt{\tau} |\xi|^q + \sqrt{\tau} |\eta|^q} \\ 
&\le C |\xi-\eta| (1+ |\eta|^{q/4}+ |\xi|^{q/4}).
\end{align*} 
thereby showing \eqref{ftau-}, in combination with Young inequality. 
\end{proof}

Similarly to the proof of Lemma \ref{lm-ex-tau}, one can show that there exist many other pairs of tamed functions \(f_\tau\) and \(g_\tau\). On the other hand, in general, the coupled monotonicity assumption \eqref{fgtau'} cannot be derived under the stated hypotheses without further assumptions, as the following determinisitc example illustrates.

\begin{remark}  
Take $q=2$, $p=2$, $C_Q=1$, $g\equiv 0$ (i.e., $L_4=L_5=0$ in \eqref{g-grow+}).
Let $\xi \in \mathbb R$ and \(r_n=2^n\) for \(n\ge 1\). Choose cut-off functions
$\psi_n\in C_c^\infty ((-r_n,-r_n/2))$, $0 \le \psi_n \le 1$, 
s.t. \(\psi_n=1\) on $[- 3r_n/4, - r_n/2]$. 
% The intervals \((-r_n,-r_n/2)=(-2^n,-2^{n-1})\) are disjoint.
For a fixed large constant \(A>0\), define
\begin{align*}
f'(\xi)=-|\xi|^2-A\sum_{n\ge 1}\psi_n(\xi)(1+|\xi|^2), \quad f(\xi):=\int_0^\xi f'(s)\,ds.
\end{align*} 
Then \(f\in C^1(\mathbb R)\) and satisfies \eqref{fg-mon}-\eqref{f-grow} 
with $L_0=0$, \(L_1=0\), \(L_2=1\), and \(L_3=A+1\).  

Now fix \(n\ge 1\) and put
$\tau_n=r_n^{-4}=2^{-4n} \in (0,1)$.
At \(x=-r_n\), the cut-off supports give
$\sum_{m\ge 1}\psi_m(-r_n)=0$, 
so
$f'(-r_n)=-r_n^2$. 
On the other hand, $f(-r_n) = -\int_{-r_n}^0 f'(s)\,ds \ge c A r_n^3$
for some constant \(c>0\).
As $f_{\tau}(\xi)= f(\xi) (1+\tau |\xi|^4)^{-1/2}$,
$f_{\tau}'(\xi)
=
\frac{f'(\xi)}{(1+\tau |\xi|^4)^{1/2}}
-
\frac{2\tau x^3 f(\xi)}{(1+\tau |\xi|^4)^{3/2}}.$
At \(x=-r_n\) and \(\tau=\tau_n=r_n^{-4}\), we have \(1+\tau_n r_n^4=2\). Hence
$f_{\tau_n}'(-r_n)
= - r_n^2/\sqrt2 +
\tau_n r_n^3 f(-r_n)/\sqrt2$.
Since \(\tau_n r_n^3=r_n^{-1}\) and \(f(-r_n)\ge c A r_n^3\),
we have $f_{\tau_n}'(-r_n)
\ge (cA-1)r_n^2/\sqrt2$.
Choosing \(A\) large enough to ensure \(cA>1\), so
$f_{\tau_n}'(-r_n)\ge C r_n^2$ for some constant \(C>0\). 
Consequently,
$2f_{\tau_n}'(-r_n)
+ \tau_n |f_{\tau_n}'(-r_n)|^2
\ge 2C r_n^2 \longrightarrow \infty$ as $n \to \infty$.
This is a contraction, since \eqref{fgtau'} would require
$2f_{\tau_n}'(-r_n)
+ \tau_n |f_{\tau_n}'(-r_n)|^2
\le M_0$, 
for some \(M_0\) independent of \(n\).  
\end{remark}

Next, we show that Example \ref{ex} satisfies Assumption \ref{ap-tau'}.

\begin{example} 
Let $q=2$, \(f(\xi)=\xi-\xi^3\), \(g(\xi)=\gamma_0 |\xi|^{1/2}\xi\), 
$f_\tau(\xi)=f(\xi)(1+\tau|\xi|^4)^{-1/2}$, and $g_\tau(\xi)=g(\xi) (1+\sqrt{\tau}|\xi|^2)^{-1/2}$, $\xi \in \mathbb R$.
Then 
$f'(\xi)=1-3\xi^2$, $|g'(\xi)|^2=9 \gamma_0^2 |\xi|/4$.
% Because \(f\) and \(g\) are odd and the taming denominators depend only on \(|\xi|\), the functions \(f_\tau'\) and \(g_\tau'\) are even. Hence, 
It suffices to consider \(x :=|\xi| \ge 0\).
Put
$\alpha=\tau^{-1/2}\ge 1$ and $y=\tau^{1/4} x \ge0$. 
Then
$1+\tau x^4=1+y^4$.
For \(x\ge0\),
\[
f_\tau'(x)
=
\frac{1-3x^2}{(1+\tau x^4)^{1/2}}
-\frac{2\tau x^3(x-x^3)}{(1+\tau x^4)^{3/2}}
=A(y)-\alpha B(y),
\] 
where
$A(y)=(1-y^4)(1+y^4)^{-3/2}$ and  
$B(y)=y^2(3+y^4)(1+y^4)^{-3/2}$.
Similarly, 
\[
g_\tau'(\xi)
=
\frac{3\gamma_0 x^{1/2}}{2(1+\sqrt{\tau}x^2)^{1/2}}
-
\frac{\gamma_0 \sqrt{\tau}x^{5/2}}{(1+\sqrt{\tau}x^2)^{3/2}}
=
\frac{\gamma_0}{2} \alpha^{1/4} \sqrt{H(y)},
\] 
where $H(y)=y(3+y^2)^2 (1+y^2)^{-3}$.
Therefore
$|g_\tau'(\xi)|^2
= \gamma_0^2 \alpha^{1/2}H(y)/4$.

Let $K=C_Q (p^*-1)\gamma_0^2/4$.
Direct computations yield \eqref{fgtau'}:
\begin{align*}
& 2f_\tau'(x)+\tau |f_\tau'(x)|^2+ C_Q (p^*-1) |g_\tau'(x)|^2 \\
% &= 2(A-\alpha B)+\alpha^{-2}(A-\alpha B)^2 + K\alpha^{1/2}H \\
% &= 2A(y)-2\alpha B(y)+ (A(y)/ \alpha-B(y))^2 + K \alpha^{1/2} H(y) \\
& \le \sup_{\alpha \ge 1, ~ y \ge 0} \{2A(y)-2\alpha B(y)+ (A(y)/ \alpha-B(y))^2
+ K \alpha^{1/2} H(y\} \\
& = \max\{3, 2+ 27 C_Q^2 (p^*-1)^2 \gamma_0^4/128\}:=M_0.
\end{align*} 
Moreover, $M_0 < 2 \lambda_1$ if and only if $27 C_Q^2 (p-1)^2 \gamma_0^4 < 256 (\lambda_1-1)$, which corresponds to the uniform condition \eqref{con-uni}, as shown in Example \ref{ex}.
\end{example}

\subsection{Auxiliary Process and Moment Estimates} 
\label{sec4.3}

% Since the drift-GTEM \eqref{t-fem} is linearly implicit, unlike the finite-D case studied in \cite{LW26a}, there seems to be no corresponding continuous-time interpolation s.t. it is an It\^o process. Consequently, one could not utilize the tool of the It\^o formula.
As in \cite{LQ21}, we define the auxiliary process $\{{\widehat Y}_n^h\}$ by ${\widehat Y}_0^h=\mathcal P_h X_0$ and  
\begin{align}\label{aux}
{\widehat Y}_n^h 
=S_{h,\tau}^n \mathcal P_h X_0
+\tau \sum_{i=0}^{n-1} S_{h,\tau}^{n-i} \mathcal P_h F_\tau(X(t_i))
+\sum_{i=0}^{n-1} S_{h,\tau}^{n-i} \mathcal P_h G_\tau(X(t_i)) \delta_i W,
\end{align}
for $n \in \mathbb N_+$, where the terms $Y_i^h$ in the discrete deterministic and stochastic convolutions of \eqref{full-sum} are both replaced by $X(t_i)$.
It is clear that for $n \in \mathbb N_+$,
\begin{align} \label{aux+} 
{\widehat Y}_n^h
&={\widehat Y}_{n-1}^h+\tau A_h {\widehat Y}_n^h 
+\tau \mathcal P_h F_\tau(X(t_{n-1}))
+\mathcal P_h G_\tau(X(t_{n-1})) \delta_{n-1} W.
\end{align}

Similarly to the smoothing property \eqref{ana} of the heat semigroup, for the discrete semigroup \(S_{h,\tau}\), there holds that 
\begin{align}  \label{sht}
|S_{h,\tau}^k P_h u\|_\mu \le C t_k^{-\frac{\mu-\nu}2} (1+ \tau \lambda^*)^{-k} \|u\|_\nu, \quad & \forall~ k \in \mathbb N_+, ~ 0 \le \nu \le \mu \le 2,~ u \in \dot H^\nu,
\end{align}   
where $\lambda^*: = \inf_{h>0} \lambda_1^h$ with $\lambda_1^h$ being the first eigenvalue of the discrete negative Dirichlet Laplacian $A_h$. 
Due to the quasi-uniformity of the grid and the Dirichlet boundary condition (DBC), the smallest discrete eigenvalue has a positive lower bound: $\lambda^*>0$. 
Taking $u=(-A_h)^{-1/2} v$ with $v \in \dot H^{-1}$ and $(\mu, \nu)=(1, 0)$, in combination with the fact that $\|A_h^{-1/2} P_h v\| \le \|v\|_{-1}$, implies 
\begin{align} \label{sht-1} 
\|S_{h,\tau}^k \mathcal P_h v\| \le C t_k^{-1/2} (1+ \tau \lambda^*)^{-k} \|v\|_{-1},
& \quad \forall~ k \in \mathbb N_+, ~ v \in \dot H^{-1}.
\end{align}

We have the following uniform-in-time moment estimate for the auxiliary process \eqref{aux} in the \(\dot H^1\)-norm, with a logarithmic factor. This factor can be removed in the 1-D case under the \(\dot H^{1+\gamma}\)-norm for any \(\gamma\in[0,1)\). These moment estimates are crucial for the error analysis in the next two sections.

\begin{lemma} \label{lm-reg-aux} 
Let $p \ge 2$, $X_0 \in L_\omega^{p(q+1)} \dot H^1$, and Assumptions \ref{ap} and (\eqref{fgtau-grow} in) \ref{ap-tau} hold.  
There exist positive constants $\tau_3 \in (0, 1]$ and $C$ s.t. for all $\tau \in (0, \tau_3)$,  
\begin{align}  \label{reg-aux}  
 \sup_{1 \le n \le N} \|{\widehat Y}_n^h\|_{L_\omega^{p} \dot H^1}  
& \le e^{CT} (1+\|X_0\|_{L_\omega^{p(q+1)} \dot H^1}^{q+1}) |\ln \tau |.
\end{align}   
Assume furthermore that $d=1$ and $X_0 \in L_\omega^{p(q+1)} \dot H^{1+\gamma}$ with $\gamma \in [0, 1)$, then 
\begin{align}  \label{reg-aux+}  
 \sup_{1 \le n \le N} \|{\widehat Y}_n^h\|_{L_\omega^{p} \dot H^{1+\gamma}}  
& \le e^{CT} (1+\|X_0\|_{L_\omega^{p(q+1)} \dot H^1}^{q+1}
+\|X_0\|_{L_\omega^p \dot H^{1+\gamma}}).
\end{align}    
Moreover, under \eqref{con-uni}, the estimates \eqref{reg-aux} and \eqref{reg-aux+} are uniform in time.
\end{lemma}

\begin{proof} 
Minkowski inequality yields that 
\begin{align*}
\|{\widehat Y}_n^h\|_{L_\omega^{p} \dot H^1} 
& = \|S_{h,\tau}^n \mathcal P_h X_0\|_{L_\omega^{p} \dot H^1}
+ \tau \sum_{i=0}^{n-1} \|S_{h,\tau}^{n-i} \mathcal P_h F_\tau(X(t_i))\|_{L_\omega^{p} \dot H^1} \\
& \quad + \Big\|\sum_{i=0}^{n-1} S_{h,\tau}^{n-i} \mathcal P_h G_\tau(X(t_i)) \delta_i W \Big\|_{L_\omega^{p} \dot H^1}.
\end{align*}
By \eqref{sht}, \eqref{fgtau-grow}, and the embedding \eqref{emb}, we get a bound for the first two terms:  
\begin{align*}
& \|S_{h,\tau}^n \mathcal P_h X_0\|_{L_\omega^{p} \dot H^1}
+ \tau \sum_{i=0}^{n-1} \|S_{h,\tau}^{n-i} \mathcal P_h F_\tau(X(t_i))\|_{L_\omega^{p} \dot H^1} \\
& \le C \|X_0\|_{L_\omega^{p} \dot H^1}
+\tau \sum_{i=0}^{n-1} t_{n-i}^{-1/2}e^{-c t_{n-i}}  \|F_\tau(X(t_i))\|_{L_\omega^{p} L_\xi^2} \\
& \le C \|X_0\|_{L_\omega^{p} \dot H^1}
+ C (1 + \|X\|^{q+1}_{L_t^\infty L_\omega^{p(q+1)} \dot H^1} ) \Big(\tau  \sum_{k=1}^\infty t_k^{-1/2}e^{-c t_k}\Big).
\end{align*} 

For the last term, discrete BDG and Minkowski inequalities and \eqref{sht} imply
\begin{align*}
& \sup_{1 \le n \le N} \mathbb E \Big\|\sum_{i=0}^{n-1} S_{h,\tau}^{n-i} \mathcal P_h G_\tau(X(t_i)) \delta_i W \Big\|_1^p \\
% & \le C \sup_{1 \le n \le N} \mathbb E ( \sum_{i=0}^{n-1} \|S_{h,\tau}^{n-i} \mathcal P_h G_\tau(X(t_i)) \delta_i W\|_1^2 )^{p/2} \\
% & \le C \sup_{1 \le n \le N}  \mathbb E ( \sum_{i=0}^{n-1} t_k^{-1} e^{-2 c t_k}  \|G_\tau(X(t_i)) \delta_i W\|^2 )^p \\
& \le C \sup_{1 \le n \le N} \Big( \sum_{i=0}^{n-1} t_{n-i}^{-1} e^{-2 c t_{n-i}} 
\|G_\tau(X(t_i)) \delta_i W\|_{L_\omega^{p} L_\xi^2}^2 \Big)^{p/2}.
\end{align*}
The continuous BDG inequality and the growth condition \eqref{fgtau-grow} yield that 
\begin{align*}
\mathbb E  \|G_\tau(X(t_i)) \delta_i W\|^p
\le C \tau^{p/2} \mathbb E  \|G_\tau(X(t_i))\|_{\mathcal L_2^0}^p
\le C \tau^{p/2} (1+ \mathbb E  \|X\|_1^{p(q+1)}),
\end{align*} 
and thus 
\begin{align*}
& \sup_{1 \le n \le N} \mathbb E \Big\|\sum_{i=0}^{n-1} S_{h,\tau}^{n-i} \mathcal P_h G_\tau(X(t_i)) \delta_i W \Big\|_1^p \\
& \le C (1 + \|X\|_{L_t^\infty L_\omega^{p(q+1)} \dot H^1}^{p(q+1)} ) 
\Big( \tau \sum_{k=1}^\infty t_k^{-1} e^{-2 c t_k} \Big)^{p/2}.
\end{align*}
Combining the above estimates, we obtain 
\begin{align*}
\sup_{0 \le n \le N} \|{\widehat Y}_n^h\|_{L_\omega^{p} \dot H^1} 
& \le C \|X_0\|_{L_\omega^{p} \dot H^1}
+ C (1 + \|X\|^{q+1}_{L_t^\infty L_\omega^{p(q+1)} \dot H^1} ) 
\Big( \tau\sum_{n=1}^\infty t_n^{-1/2}e^{-c t_n} \Big) \\
& \quad + C (1 + \|X\|_{L_t^\infty L_\omega^{2 p(q+1)} \dot H^1}^{2p(q+1)} ) 
\Big( \tau \sum_{n=1}^\infty t_n^{-1} e^{-2 c t_n} \Big)^{1/2}.
\end{align*} 
Using the facts that 
$\sum_{n=1}^\infty n^{-1/2}e^{-c \tau n}
 \le C \tau^{-1/2}$ and %  \sqrt{\pi/c}, 
 there exists $\tau_3 \in (0, 1]$ s.t. 
$\sum_{n=1}^\infty n^{-1} e^{- c \tau n} 
% = -\ln (1-e^{-c \tau}) \le |\ln \tau| + \ln \frac{1+c}{c}
 \le C |\ln \tau|$ for all $\tau \in (0, \tau_3)$,  
we have 
\begin{align*}
\sup_{n \in \mathbb N} \|{\widehat Y}_n^h\|_{L_\omega^{p} \dot H^1} 
& \le C \|X_0\|_{L_\omega^{p} \dot H^1}
+ C (1 + \|X\|^{q+1}_{L_t^\infty L_\omega^{p(q+1)} \dot H^1}) |\ln \tau|. 
\end{align*} 
This, in combination with the uniform regularity \eqref{spde-hol}, gives \eqref{reg-aux}.    

In the case $d=1$ and $X_0 \in L_\omega^{p(q+1)} \dot H^{1+\gamma}$ with $\gamma \in [0, 1)$, by \eqref{sht}, \eqref{fgtau-grow}, and the embedding \eqref{emb}, we get  
\begin{align*}
& \|S_{h,\tau}^n \mathcal P_h X_0\|_{L_\omega^{p} \dot H^{1+\gamma}}
+ \tau \sum_{i=0}^{n-1} \|S_{h,\tau}^{n-i} \mathcal P_h F_\tau(X(t_i))\|_{L_\omega^{p} \dot H^{1+\gamma}} \\
& \le C \|X_0\|_{L_\omega^{p} \dot H^{1+\gamma}} 
+ C (1 + \|X\|^{q+1}_{L_t^\infty L_\omega^{p(q+1)} \dot H^1} ) 
\Big(\tau  \sum_{k=1}^\infty t_k^{-(1+\gamma)/2}e^{-c t_k} \Big).
\end{align*}
For the last term, by discrete BDG, Minkowski, and H\"older inequalities, \eqref{sht}, \eqref{fgtau-grow}, and \eqref{G21}, we have 
\begin{align*}
& \sup_{1 \le n \le N} \Big\|\sum_{i=0}^{n-1} S_{h,\tau}^{n-i} \mathcal P_h G_\tau(X(t_i)) \delta_i W \Big\|_{L_\omega^{p} \dot H^{1+\gamma}} \\
% & \le C \sup_{1 \le n \le N} \mathbb E ( \sum_{i=0}^{n-1} \|S_{h,\tau}^{n-i} \mathcal P_h G_\tau(X(t_i)) \delta_i W\|_{1+\gamma}^2 )^{1/2} \\ 
& \le C \sup_{1 \le n \le N} \Big( \sum_{i=0}^{n-1} t_i^{-\gamma} e^{-2 c t_i} \|G_\tau(X(t_i)) \delta_i W\|_{L_\omega^{p} \dot H^1}^2 \Big)^{1/2} \\
& \le C (1 + \|X\|_{L_t^\infty L_\omega^{p(q/4+1)} \dot H^1}^{q/4+1} ) 
\Big( \tau \sum_{k=1}^\infty t_k^{-\gamma} e^{-2 c t_k} \Big)^{1/2}.
\end{align*} 
Combining the above two estimates and using the fact that 
$\sum_{n=1}^\infty n^{- \beta}e^{-c \tau n}
 \le C \tau^{- \beta}$ for all $\beta \in [0, 1)$, we obtain 
\begin{align*}
\sup_{0 \le n \le N} \|{\widehat Y}_n^h\|_{L_\omega^{p} \dot H^1} 
& \le C \|X_0\|_{L_\omega^{p} \dot H^{1+\gamma}} 
+ C (1 + \|X\|^{q+1}_{L_t^\infty L_\omega^{p(q+1)} \dot H^1}).
\end{align*} 
Then we conclude \eqref{reg-aux+} by \eqref{est-spde}. 
\end{proof}

\begin{remark}  \label{rk-aux}
To derive the \(L_t^\infty L_\omega^p L_\xi^\infty\)- or \(L_t^\infty L_\omega^p \dot H^{1+\gamma}\)-regularity of the auxiliary process \eqref{aux}, with \(\gamma\in(0,1)\), in the 2-D and 3-D cases, one has to estimate the \(L_t^\infty L_\omega^{2p}\mathcal L_2^1\)-norm of \(G_\tau(X)\). This seems impossible, as shown in \eqref{G2+}, where only the boundedness of
$\|G(X)\|_{L_t^{2\alpha_1}L_\omega^p\mathcal L_2^1}$
for finite \(\alpha_1\) is available. This is the main reason why the strong error estimate in Theorem \ref{tm-err} cannot be established for \(d=2,3\). 
\end{remark}

\subsection{Strong Error Estimate for Auxiliary Process}
\label{sec4.4}

In this part, we establish the strong error estimate between the exact solution $X$ of Eq.~\eqref{see} and the auxiliary process $\{\widehat Y_n^h\}$ defined in \eqref{aux} for $d=1,2,3$.  

Let $N$ be a fixed positive integer.
Denote by
$E_{h,\tau}(t)=S(t)-S_{h,\tau}^n\mathcal P_h$,
$t\in(t_{n-1},t_n]$, $1\le n\le N$.
We will use the following well-known estimate of \(E_{h,\tau}\); see, e.g., \cite[Lemma 4.1]{LQ21}:
\begin{align} \label{eht}  
\|E_{h,\tau} (t) x\| 
\le C (h^\mu+\tau^\frac\mu2) e^{-ct} t^{-\frac{\mu-\nu}2} \|x\|_\nu,
& \quad \forall~t>0, ~ x \in \dot H^\nu.
\end{align}

\begin{theorem} \label{tm-aux} 
Let $p \ge 2$, $X_0 \in L_\omega^{2p(q+1)} \dot H^1$, and Assumptions \ref{ap} and (\eqref{f-ftau} in) \ref{ap-tau} hold. 
For any $\epsilon \in (0, 1)$, there exists a positive constant $C$ s.t.    
\begin{align} \label{err-aux}
& \sup_{0 \le n \le N} \|X(t_n)-{\widehat Y}_n^h\|_{L_\omega^p L_\xi^2}
 \le e^{CT} (1 + \|X_0\|_{L_\omega^{2p(q+1)} \dot H^1}^{2q+1})
(h^{1-\epsilon}+\tau^{(1-\epsilon)/2}).
\end{align}   
Assume furthermore that $d=1$ and $X_0 \in L_\omega^{2p(q+1)} \dot H^{1+\gamma}$ with $\gamma \in [0, 1)$, then 
\begin{align}  \label{err-aux+}  
& \sup_{0 \le n \le N} \|X(t_n)-{\widehat Y}_n^h\|_{L_\omega^p L_\xi^2}  \nonumber \\ 
& \le e^{CT} (1+\|X_0\|_{L_\omega^{p(q+1)} \dot H^1}^{q+1}
+\|X_0\|_{L_\omega^p \dot H^{1+\gamma}})(h^{1+\gamma}+\tau^{(1+\gamma)/2}).
\end{align}    
Moreover, under \eqref{con-uni}, the estimates \eqref{err-aux} and \eqref{err-aux+} are uniform in time. 
\end{theorem}

\begin{proof}
Let $p \ge 2$ and $1 \le n \le N$.
Subtracting ${\widehat Y}_n^h$ in \eqref{aux} from the mild formulation \eqref{mild}
with $t=t_n$, we get
\begin{align} \label{j}
J^n  : = & \|X(t_n )-{\widehat Y}_n^h \|_{L_\omega^p L_\xi^2}
\le \| E_{h,\tau}(t_n ) X_0\|_{L_\omega^p L_\xi^2} \nonumber \\
& + \Big\|\sum_{i=0}^{n-1} \int_{t_i}^{t_{i+1}} [S(t_n-r) F(X(r))-S_{h,\tau}^{n-i} \mathcal P_h F_\tau(X(t_i))] \,\mathrm{d}r \Big\|_{L_\omega^p L_\xi^2}\nonumber \\
& + \Big\|\sum_{i=0}^{n-1} \int_{t_i}^{t_{i+1}} [S(t_n-r) G(X(r))-S_{h,\tau}^{n-i} \mathcal P_h G_\tau(X(t_i))] \,\mathrm{d}W \Big\|_{L_\omega^p L_\xi^2} 
=: \sum_{j=1}^3 J^n_j.
\end{align} 
In the sequel, we treat the above three terms one by one.

The estimate \eqref{eht} with $\mu=\nu=1$ yields that 
\begin{align}  \label{j1}
J^n_1
\le C e^{-c t_n} (h+\tau^{1/2} ) \|X_0\|_{L_\omega^p \dot H^1}.
\end{align}
To handle the second term, we decompose it into the following three terms:
\begin{align*} 
J^n  _2 
&\le \sum_{i=0}^{n-1} \int_{t_i}^{t_{i+1}} 
\|S(t_n-r) [F(X(r))-F(X(t_i))] \|_{L_\omega^p L_\xi^2} \,\mathrm{d}r  \\
&\quad + \sum_{i=0}^{n-1} \int_{t_i}^{t_{i+1}} 
\|E_{h,\tau}(t_n-r) F(X(t_i)) \|_{L_\omega^p L_\xi^2} \,\mathrm{d}r \\
&\quad + \sum_{i=0}^{n-1} \int_{t_i}^{t_{i+1}} 
\|S_{h,\tau}^{n-i} \mathcal P_h [F(X(t_i))-F_\tau(X(t_i))]\|_{L_\omega^p L_\xi^2} \,\mathrm{d}r 
=: \sum_{i=1}^3 J^n  _{2i}.
\end{align*} 
The bound \eqref{ana} with $(\mu, \nu)=(1, 0)$ and the dual estimate \eqref{F-} yield
\begin{align*} 
J^n  _{21} 
&\le C \sum_{i=0}^{n-1} \int_{t_i}^{t_{i+1}} (t_n-r)^{-1/2} e^{-c (t_n-r)}
\|F(X(r))-F(X(t_i)) \|_{L_\omega^p \dot H^{-1}} \,\mathrm{d}r \\ 
% &\le C \sum_{i=0}^{n-1} \int_{t_i}^{t_{i+1}} (t_n-r)^{-1/2} e^{-C (t_n-r)} \| (1+\|X(r)\|^q_1+\|X(t_i)\|^q_1 ) \|X(r)-X(t_i)\|\|_{L_\omega^p} \,\mathrm{d}r \\
% & \le C \tau^{1/2}  (1+\|X\|^q_{L_t^\infty L_\omega^{2pq} \dot H^1} ) \|X\|_{\mathcal C_t^{1/2} L_\omega^{2p} L_\xi^2} \cdot (\int_0^\infty r^{-1/2} e^{-C r} \,\mathrm{d}r   ) \\  
& \le C \tau^{1/2} (1+\|X\|^q_{L_t^\infty L_\omega^{2pq} \dot H^1} )
\|X\|_{\mathcal C_t^{1/2} L_\omega^{2p} L_\xi^2}, \quad \gamma\in [0,1],
\end{align*}  
where the elementary estimate $\int_0^\infty r^{-1/2} e^{-c r} \,\mathrm{d}r \le C<\infty$ is used.
Similarly, using \eqref{eht} with $(\mu, \nu) = (1, 0)$ and the embedding \eqref{emb}, we derive  
\begin{align} \label{j22}
J^n_{22}
% &\le C (h+\tau^\frac{1+\gamma}2)  \|F(X)\|_{L_t^\infty L_\omega^p L_\xi^2} (\int_0^{t_n} g^{-\frac{1+\gamma}2} e^{-C g} \,\mathrm{d}g)  \nonumber \\
&\le C (h+\tau^{1/2}) 
(1+\|X\|_{L_t^\infty L_\omega^p \dot H^1}^{q+1}).
\end{align}  
By \eqref{sht-1}, \eqref{F-Ftau}, and the estimate 
$\sum_{k=1}^\infty k^{-1/2} (1+ \tau \lambda^*)^{- k}
   \le C \tau^{-1/2}$,
we get  
\begin{align} \label{j23} 
J^n  _{23} 
& \le \|F(X)-F_\tau(\xi)\|_{L_t^\infty L_\omega^p \dot H^{-1}}
\Big( \tau^{1/2} \sum_{k=1}^\infty k^{-1/2} (1+ \tau \lambda^*)^{- k} \Big) 
\nonumber  \\
% &\le C \tau (1+\|X\|_{L_t^\infty L_\omega^{p(3q+1)} L_\xi^\infty}^{3q+1}) \sum_{i=1}^n\tau e^{-c i \tau}\nonumber \\
&\le C \tau^{1/2} (1+\|X\|_{L_t^\infty L_\omega^{p(2q+1)} \dot H^1}^{2q+1}).
\end{align}  
Combining the above three estimates and Theorem \ref{tm-reg+} implies  
\begin{align}  \label{j2}
J^n_2
& \le C (h+\tau^{1/2}) \Big[(1+\|X\|^q_{L_t^\infty L_\omega^{2pq} \dot H^1} ) \|X\|_{\mathcal C_t^{1/2} L_\omega^{2p} L_\xi^2} \nonumber \\ 
& \quad +(1+\|X\|_{L_t^\infty L_\omega^p \dot H^1}^{q+1})+ \|X\|_{L_t^\infty L_\omega^{p(2q+1)} \dot H^1}^{2q+1} \Big]  \nonumber \\ 
& \le C (h+\tau^{1/2} )  
(1+ \|X_0\|_{L_\omega^{2p(q+1)} \dot H^1}^{2q+1}).
\end{align}

The last term $J^n_3$ can be controlled by using BDG inequalities: 
\begin{align*} 
(J^n_3)^2  
&\le \sum_{i=0}^{n-1} \int_{t_i}^{t_{i+1}} 
\|S(t_n-r) [G(X(r))-G(X(t_i))] \|_{L_\omega^p \mathcal L_2^0}^2 \,\mathrm{d}r  \\
& \quad + C \sum_{i=0}^{n-1} \int_{t_i}^{t_{i+1}} 
\|E_{h,\tau}(t_n-r) G(X(t_i)) \|_{L_\omega^p \mathcal L_2^0}^2 \,\mathrm{d}r \\
& \quad + C \sum_{i=0}^{n-1} \int_{t_i}^{t_{i+1}} 
\|S_{h,\tau}^{n-i} \mathcal P_h [G(X(t_i))-G_\tau(X(t_i))]\|_{L_\omega^p \mathcal L_2^0}^2 \,\mathrm{d}r 
:=\sum_{j=1}^3 (J^n_{3j})^2.
\end{align*} 
For each $q>0$, we have $8(q+1)/(3q+4)<6$ and thus there exists a constant $\gamma \in (0, 1)$ s.t. $\dot H^\gamma \hookrightarrow L_\xi^{q_1^*}$  (as $\dot H^1 \hookrightarrow L_\xi^6$) and $L_\xi^{q_1} \hookrightarrow \dot H^{-\gamma}$ with $q_1:=8(q+1)/(5q+4)$.
Using \eqref{ana} with $(\mu, \nu)=(1, 0)$, \eqref{cq}, \eqref{g-grow+}, and \eqref{est-spde}, we get 
{
\begin{align*} 
(J^n_{31})^2
&\le C \sum_{i=0}^{n-1} \int_{t_i}^{t_{i+1}} (t_n-r)^{-\gamma} e^{-c (t_n-r)} 
\|G(X(r))-G(X(t_i))\|_{L_\omega^p \mathcal L_2^{-\gamma}}^2 \,\mathrm{d}r  \nonumber  \\ 
&\le C \sum_{i=0}^{n-1} \int_{t_i}^{t_{i+1}} (t_n-r)^{-\gamma} e^{-c (t_n-r)} \\
& \qquad \times \|[1+|X(r)|^{q/4}+|X(t_i)|^{q/4}] |X(r)-X(t_i)| \|_{L_\omega^p L_\xi^{q_1}}^2 \,\mathrm{d}r  \nonumber  \\ 
&\le C \sum_{i=0}^{n-1} \int_{t_i}^{t_{i+1}} (t_n-r)^{-\gamma} e^{-c (t_n-r)} \\
& \qquad \times \| [1+\|X(r)\|_{L_\xi^{2(q+1)}}^{q/4}
+ \|X(t_i)\|_{L_\xi^{2(q+1)}}^{q/4} ] \|X(r)-X(t_i)\| \|^2_{L_\omega^p} \,\mathrm{d}r  \nonumber  \\ 
&\le C \tau (1+\|X\|_{L_t^\infty L_\omega^{p q/2} \dot H^1}^{q/4})^2
\|X\|_{\mathcal C_t^{1/2} L_\omega^{2 p} L_\xi^2}^2. 
\end{align*} 
} 
For the second term $J^n_{32}$, we use \eqref{eht} with $(\mu,\nu)=(1-\epsilon, 0)$ and \eqref{g-grow} to derive 
\begin{align*} 
(J^n_{32})^2
&\le C (h^{1-\epsilon}+\tau^{(1-\epsilon)/2})^2 
(1+ \|G(X)\|_{L_t^\infty L_\omega^p \mathcal L_2^0})^2
 \Big(\int_0^\infty r^{-(1-\epsilon)}  e^{-2c r} \,\mathrm{d}r\Big) \nonumber \\
 & \le C (h^{1-\epsilon}+\tau^{(1-\epsilon)/2})^2 
(1 + \|X\|_{L_t^\infty L_\omega^{p(q/4+1)} \dot H^1}^{q/4+1} )^2.
\end{align*} 
For the last term $J^n_{33}$, we use \eqref{sht-1}, the embedding \eqref{emb},  \eqref{G-Gtau}, and the elementary inequality 
$\sum_{k=1}^\infty k^{-1} (1+ \tau \lambda^*)^{-2 k}
 % = -\ln [1- (1+\tau \lambda^*)^{-2}]  
 \le C |\ln \tau|$,
to derive  
\begin{align*} 
(J^n  _{33})^2 
& \le \|G(X)-G_\tau(X)\|^2_{L_t^\infty L_\omega^p \mathcal L_2^{-1}} 
\Big(\sum_{k=1}^\infty k^{-1} (1+ \tau \lambda^*)^{-2 k} \Big)  \nonumber  \\
% &\le C \tau (1+\|X\|_{L_t^\infty L_\omega^{p(3q+1)} L_\xi^\infty}^{3q+1}) \sum_{i=1}^n\tau e^{-c i \tau}\\
% &\le C \tau \|g(X(t_i))-g_\tau(X(t_i))\|_{L_t^\infty L_\omega^p L_\xi^{2(q+1)/(2q+1)}}^2 ( \tau \sum_{k=1}^\infty t_k^{-1} (1+ \tau \lambda^*)^{-2 k}  ) \nonumber  \\
&\le C \tau |\ln \tau| (1+ \|X\|_{L_t^\infty L_\omega^{p(2q+1)} \dot H^1}^{2q+1})^2.
\end{align*}

It follows from the estimate \eqref{est-spde} that 
\begin{align}  \label{j3} 
J^n_3
&\le C (h+\tau^{1/2})  \Big[(1+\|X\|_{L_t^\infty L_\omega^{p q/2} \dot H^1}^{q/4} )\|X\|_{\mathcal C_t^{1/2} L_\omega^{2 p} L_\xi^2} 
+ \|X\|_{L_t^\infty L_\omega^{p(2q+1)} \dot H^1}^{2q+1} \Big]  \nonumber  \\
& \le C (1 + \|X_0\|_{L_\omega^{2p(q+1)} \dot H^1}^{2q+1}).
\end{align} 
Putting the estimates \eqref{j1}, \eqref{j2}, and \eqref{j3} together and noting that the term $|\ln \tau|$ can be absored in the term $\tau^{-\epsilon/2}$ results in
\begin{align*} 
J^n  \le C (1 +\|X_0\|_{L_\omega^{2p(q+1)} \dot H^1}^{2q+1}) (h^{1-\epsilon}+\tau^{(1-\epsilon)/2}).
\end{align*}
Combining this inequality with the estimate  
$\|X_0-\mathcal P_h X_0\|_{L_\omega^p L_\xi^2} \le C h \|X_0\|_{L_\omega^p \dot H^1}$, completes the proof of \eqref{err-aux}.  

To prove \eqref{err-aux+}, it suffices, by the preceding proof, to provide refined estimates for the terms \(J^n_1\), \(J^n_{22}\), \(J^n_{32}\), and \(J^n_{33}\), involving either the spatial convergence rates or the removal of the factor \(|\ln\tau|\), when $d=1$ and $X_0 \in L_\omega^{2p(q+1)} \dot H^{1+\gamma}$ with $\gamma \in [0, 1)$.
At first, the estimate \eqref{eht} with $\mu=\nu=1+\gamma$ yields that 
\begin{align*}
J^n_1
\le C e^{-c t_n} (h^{1+\gamma}+\tau^{(1+\gamma)/2} ) \|X_0\|_{L_\omega^p \dot H^{1+\gamma}}.
\end{align*}
Using \eqref{eht} with $(\mu, \nu) = (1+\gamma, 0)$ and the embedding \eqref{emb}, we derive  
\begin{align*}
J^n_{22}
&\le C (h^{1+\gamma}+\tau^{(1+\gamma)/2})  \|F(X)\|_{L_t^\infty L_\omega^p L_\xi^2} \Big(\int_0^\infty r^{-\frac{1+\gamma}2} e^{-c r} \,\mathrm{d}r \Big)  \nonumber \\
&\le C (h^{1+\gamma}+\tau^{(1+\gamma)/2}) 
(1+\|X\|_{L_t^\infty L_\omega^p \dot H^1}^{q+1}).
\end{align*}  
For the term $J^n_{32}$, we use \eqref{eht} with $(\mu,\nu)=(1+\gamma, 0)$ and \eqref{G21} to derive 
\begin{align*} 
(J^n_{32})^2
&\le C (h^{1+\gamma}+\tau^{(1+\gamma)/2})^2 
(1+ \|G(X)\|_{L_t^\infty L_\omega^p \mathcal L_2^1})^2
 \Big(\int_0^\infty r^{-\gamma}  e^{-2c r} \,\mathrm{d}r\Big) \nonumber \\
 & \le C (h^{1+\gamma}+\tau^{(1+\gamma)/2})^2 
(1 + \|X\|_{L_t^\infty L_\omega^{p(q/4+1)} \dot H^1}^{q/4+1} )^2.
\end{align*} 
For the last term $J^n_{33}$, we use \eqref{sht-1}, \eqref{f-ftau}, the embedding $\dot H^1 \hookrightarrow L_\xi^\infty$,  and the elementary inequality 
$\sum_{k=1}^\infty k^{-1} (1+ \tau \lambda^*)^{-2 k}
 % = -\ln [1- (1+\tau \lambda^*)^{-2}]  
 \le C |\ln \tau|$,
to derive  
\begin{align*} 
(J^n  _{33})^2 
& \le \|G(X)-G_\tau(X)\|^2_{L_t^\infty L_\omega^p \mathcal L_2^0} 
\Big(\tau \sum_{k=1}^\infty (1+ \tau \lambda^*)^{-2 k} \Big)  \nonumber  \\ 
&\le C \tau (1+ \|X\|_{L_t^\infty L_\omega^{p(2q+1)} \dot H^1}^{2q+1})^2.
\end{align*}  
This completes the proof.
\end{proof}

\subsection{Strong Error Estimate for Tamed-FEM}
\label{sec4.5}

As a byproduct of Theorem \ref{tm-aux}, we obtain the following estimates for the difference between the tamed drift and diffusion operators in the 1-D case.

\begin{corollary}  \label{cor-fgtau}
Let $d=1$, $p \ge 2$, $p_0:=\max\{2p(q+1), 2pq(q+1)\}$, $\gamma \in [0, 1)$, $X_0 \in L_\omega^{p_0} \dot H^{1+\gamma}$, and Assumptions \ref{ap} and \ref{ap-tau} hold. 
There exists a constant $C$ s.t.  
\begin{align}  \label{err-fgtau}  
& \|G_\tau(X(t_{n-1}))-G_\tau({\widehat Y}_{n-1}^h)\|_{L_\omega^p \mathcal L_2^0}
+ \|F_\tau(X(t_{n-1}))-F_\tau({\widehat Y}_{n-1}^h)\|_{L_\omega^p \dot H^{-1}} \nonumber \\
& \le e^{CT} (1+\|X_0\|_{L_\omega^{p(q+1)} \dot H^1}^{(q+1)^2}
+\|X_0\|^{(q+1)^2}_{L_\omega^p \dot H^{1+\gamma}})
(h^{1+\gamma}+\tau^{(1+\gamma)/2}), \quad \tau \in (0, \tau_3).
\end{align}    
Moreover, under \eqref{con-uni}, the estimate \eqref{err-fgtau} is uniform in time.  
\end{corollary}

 \begin{proof} 
 By \eqref{Ftau-} and the embedding $\dot H^1 \hookrightarrow L_\xi^\infty$, we have 
 \begin{align*}
& \|G_\tau(X(t_{n-1}))-G_\tau({\widehat Y}_{n-1}^h)\|_{L_\omega^p \mathcal L_2^0} 
+ \|F_\tau(X(t_{n-1}))-F_\tau({\widehat Y}_{n-1}^h)\|_{L_\omega^p \dot H^{-1}} \nonumber \\
& \le \|[1 + \|X(t_{n-1})\|_{L_\xi^\infty}^q + \|{\widehat Y}_{n-1}^h\|_{L_\xi^\infty}^q] \|X(t_{n-1})-{\widehat Y}_{n-1}^h\|]\|_{L_\omega^{2p}} \\
& \quad + \|[1 + \|X(t_{n-1})\|_1^q + \|{\widehat Y}_{n-1}^h\|_1^q] \|X(t_{n-1})-{\widehat Y}_{n-1}^h\|]\|_{L_\omega^p} \\
& \le (1 + \|X\|^q_{L_t^\infty L_\omega^{2pq} \dot H^1}
+ \sup_{0 \le n \le N} \|{\widehat Y}_n^h\|^q_{L_\omega^{2pq} \dot H^1})
\|X(t_{n-1})-{\widehat Y}_{n-1}^h\|_{L_\omega^{2p} L_\xi^2}.
\end{align*}  
  Then we conclude \eqref{err-fgtau} by the above estimate, the regularities \eqref{est-spde} and \eqref{reg-aux+}, and the error estimate \eqref{err-aux+}.
 \end{proof}

 \begin{remark}
Unlike the Lipschitz diffusion case studied in \cite{FLZ17, LQ21, LS26, CL26}, in the superlinear diffusion case considered here, one needs boundedness for the \(L_\xi^\infty\)-norm of both the exact solution to Eq.~\eqref{spde} and the corresponding auxiliary process \eqref{aux}.
In this case, to obtain higher-order spatial convergence, one has to replace the term
$\|G(X)\|_{L_t^\infty L_\omega^p \mathcal L_2^0}$
in the estimate of \(J_{32}^n\) (using \eqref{eht} with \((\mu,\nu)=(1+\gamma,1)\) and \eqref{G2}) by
$\|G(X)\|_{L_t^\infty L_\omega^p \mathcal L_2^1}$.
However, according to Remark~\ref{rk-aux}, it is not feasible to establish the \(L_t^\infty L_\omega^p L_\xi^\infty\)- or \(L_t^\infty L_\omega^p \dot H^{1+\gamma}\)-regularity of the auxiliary process \eqref{aux} in the 2-D or 3-D case, although we established such regularity for Eq.~\eqref{spde} in Theorem \ref{tm-reg+}. 
\end{remark}

By combining Lemma~\ref{tm-aux} with a variational approach, we obtain the following strong convergence rate between the solution of Eq.~\eqref{spde} and the tamed-FEM \eqref{t-fem}.

 \begin{theorem}  \label{tm-err}
Let $d=1$, $p^*>p \ge 2$, $\gamma \in [0, 1)$, $X_0 \in L_\omega^{p_0} \dot H^{1+\gamma}$, with $p_0$ given in Corollary \ref{cor-fgtau}, and Assumptions \ref{ap}, \ref{ap-tau}, and \ref{ap-tau'} hold. 
There exist positive constants $\tau_{\max} \in (0, 1]$  and $C$ s.t.  for any $\tau \in (0, \tau_{\max})$,
\begin{align} \label{err} 
& \sup_{1 \le n \le N} \|X(t_n) - Y_n^h \|_{L_\omega^p L_\xi^2} \nonumber \\
& \le e^{CT} (1+\|X_0\|_{L_\omega^{p(q+1)} \dot H^1}^{(q+1)^2}
+\|X_0\|^{(q+1)^2}_{L_\omega^p \dot H^{1+\gamma}})
(h^{1+\gamma}+\tau^{(1+\gamma)/2}).
\end{align} 
Moreover, under the conditions \eqref{con-uni} and $M_0<2 \lambda_1$, with $M_0$ given in \eqref{fgtau'}, the estimate \eqref{err} is uniform in time.
 \end{theorem}

 \begin{proof} 
 Let $0 \le n \le N$ and denote  
$e_n^h:={\widehat Y}_n^h -Y_n^h$. 
Then $e_n^h  \in V_h$ with vanishing initial datum $e^h_0=0$.
In terms of \eqref{aux+} and \eqref{t-fem}, it is clear that  
\begin{align*}
& e_n^h - e_{n-1}^h - \tau A_h e_n^h \\
& = \tau \mathcal P_h [F_\tau(X(t_{n-1}))-F_\tau(Y_{n-1}^h)]  
+\mathcal P_h [G_\tau(X(t_{n-1}))-G_\tau(Y_{n-1}^h)] \delta_{n-1} W.
\end{align*}
Testing with $e_n^h$ on both sides of the above equation and using \eqref{ab}, we obtain 
\begin{align*}
& \|e_n^h\|^2-\|e_{n-1}^h\|^2 + \|e_n^h -e_{n-1}^h\|^2 +2 \tau \|\nabla e_n^h  \|^2\\
& =  2\tau \langle F_\tau(X(t_{n-1}))-F_\tau({\widehat Y}_{n-1}^h), e_n^h\rangle _{V^*, V}
+ 2\tau \langle  F_\tau({\widehat Y}_{n-1}^h)-F_\tau(Y_{n-1}^h), e_{n-1}^h\rangle_{V^*, V} \\
 & + 2\langle e_n^h  -e_{n-1}^h, \tau [F_\tau({\widehat Y}_{n-1}^h)-F_\tau(Y_{n-1}^h)]+[G_\tau(X(t_{n-1}))-G_\tau(Y_{n-1}^h)]\delta_{n-1} W\rangle \\
 & + 2\langle e_{n-1}^h , (G_\tau(X(t_{n-1}) )-G_\tau(Y_{n-1}^h))\delta_{n-1} W\rangle .
\end{align*} 
By the Poincar\'e and Young inequalities, for any small $\epsilon>0$, we have 
\begin{align}\label{en+}
	& [1+2(\lambda_1-\epsilon)\tau]\|e_n^h\|^2
	\le \|R_{n-1}\|^2
	+ C \tau \|F_\tau(X(t_{n-1}))-F_\tau({\widehat Y}_{n-1}^h)\|_{-1}^2,
\end{align} 
with $R_{n-1}:=e_{n-1}^h+\tau \mathcal P_h [F_\tau({\widehat Y}_{n-1}^h)-F_\tau(Y_{n-1}^h)] + \mathcal P_h [G_\tau(X(t_{n-1}))-G_\tau(Y_{n-1}^h)]\delta_{n-1} W$.
Then, taking the $p/2$-th power and $\mathbb E_{n-1}[\cdot]$ and using \eqref{in-ep}, we obtain
\begin{align} \label{ehn}
	\mathbb E_{n-1} \|e^h_n\|^p
&  \le \frac{(1+\epsilon\tau)}{[1+ 2 (\lambda_1-\epsilon) \tau]^{p/2}} \mathbb E_{n-1}\|R_{n-1}\|^p + C \tau \|F_\tau(X(t_{n-1}))-F_\tau({\widehat Y}_{n-1}^h)\|_{-1}^p.
\end{align}

For $0\le t\le \tau$, set
\begin{align*}
J_{n-1}(t):& = e_{n-1}^h+\tau \mathcal P_h [F_\tau({\widehat Y}_{n-1}^h)-F_\tau(Y_{n-1}^h)] \\
& \quad + \mathcal P_h [G_\tau(X(t_{n-1}))-G_\tau(Y_{n-1}^h)] [W(t_{n-1}+t)-W(t_{n-1})].
\end{align*}
Then $J_{n-1}(0)=e_{n-1}^h+\tau \mathcal P_h [F_\tau({\widehat Y}_{n-1}^h)-F_\tau(Y_{n-1}^h)]$ and $J_{n-1}(\tau)=R_{n-1}$.  
Applying the It\^o formula \eqref{ito} to $\|J_{n-1}(t)\|^p$, followed by taking $\mathbb E_{n-1}$, and utilizing the estimate  
\begin{align*}
\|[G_\tau(X(t_{n-1}))-G_\tau(Y_{n-1}^h)]^* J_{n-1}(t)\|_{U_0}
&  \le \|G_\tau(X(t_{n-1}))-G_\tau(Y_{n-1}^h)\|_{\mathcal L_2^0} \|J_{n-1}(t)\|,
\end{align*}   
and H\"older inequality, we infer
\begin{align*}
& \frac{\,\mathrm{d}}{\,\mathrm{d} t}\mathbb E_{n-1}\|J_{n-1}(t)\|^p \\
&=\frac p2\mathbb E_{n-1} [\|J_{n-1}(t)\|^{p-2}\|G_\tau(X(t_{n-1}))-G_\tau(Y_{n-1}^h)\|_{\mathcal L_2^0}^2 ] \\
& \quad + \frac12 p(p-2) \mathbb E_{n-1} [\|J_{n-1}(t)\|^{p-4}\|[G_\tau(X(t_{n-1}))-G_\tau(Y_{n-1}^h)]^*J_{n-1}(t)\|_{U_0}^2 ]\\
&\le \frac12 p(p-1) \|G_\tau(X(t_{n-1}))-G_\tau(Y_{n-1}^h)\|_{\mathcal L_2^0}^2 \mathbb E_{n-1}\|J_{n-1}(t)\|^{p-2}.
\end{align*} 
By H\"older inequality, we have
$\mathbb E_{n-1}\|J_{n-1}(t)\|^{p-2} \le (\mathbb E_{n-1}\|J_{n-1}(t)\|^p)^{1-2/p}$.
Substituting this into the above estimate yields
$$\frac{\,\mathrm{d}}{\,\mathrm{d} t} (\mathbb E_{n-1} \|J_{n-1}(t)\|^p)^{2/p} \le (p-1)\|[G_\tau(X(t_{n-1}))-G_\tau(Y_{n-1}^h)]\|_{\mathcal L_2^0}^2.$$
Integrating over $[0,\tau]$ and using Young inequality yield
\begin{align*}
& \mathbb E_{n-1} \|R_{n-1}\|^p \\
& \le (\|e_{n-1}^h+\tau \mathcal P_h [F_\tau({\widehat Y}_{n-1}^h)-F_\tau(Y_{n-1}^h)]\|^2 \nonumber \\
& \quad + (p-1) \tau \|G_\tau(X(t_{n-1}))-G_\tau(Y_{n-1}^h)\|_{\mathcal L_2^0}^2 )^{p/2} \nonumber \\
& \le (\|e_{n-1}^h+\tau \mathcal P_h [F_\tau({\widehat Y}_{n-1}^h)-F_\tau(Y_{n-1}^h)]\|^2 \nonumber \\
& \quad + (1+\epsilon_2) (p-1)\tau \|G_\tau({\widehat Y}_{n-1}^h))-G_\tau(Y_{n-1}^h)\|_{\mathcal L_2^0}^2 \nonumber \\
& \quad + C_{\epsilon_2} \tau \|G_\tau(X(t_{n-1}))-G_\tau({\widehat Y}_{n-1}^h)\|_{\mathcal L_2^0}^2 )^{p/2} \nonumber \\
& \le (\|e_{n-1}^h\|^2 + C_{\epsilon_2}  \tau \|G_\tau(X(t_{n-1}))-G_\tau({\widehat Y}_{n-1}^h)\|_{\mathcal L_2^0}^2 \\
& \quad + \tau [ 2 \langle F_\tau({\widehat Y}_{n-1}^h)-F_\tau(Y_{n-1}^h), e_{n-1}^h \rangle_{V^*, V} + \tau \|F_\tau({\widehat Y}_{n-1}^h)-F_\tau(Y_{n-1}^h)\|^2   \nonumber \\
& \quad + (1+\epsilon_2) (p-1) \|G_\tau({\widehat Y}_{n-1}^h))-G_\tau(Y_{n-1}^h)\|_{\mathcal L_2^0}^2]  )^{p/2}.
\end{align*}
Then, by \eqref{FGtau-} with $(1+\epsilon_2) (p-1) \le p^*-1$ and \eqref{in-ep}, we have 
\begin{align*}
\mathbb E_{n-1} \|R_{n-1}\|^p 
& \le ([1+ M_0 \tau] \|e_{n-1}^h\|^2
+ C_{\epsilon_2} \tau \|G_\tau(X(t_{n-1}))-G_\tau({\widehat Y}_{n-1}^h)\|_{\mathcal L_2^0}^2)^{p/2} \nonumber \\ 
& \le (1+\epsilon \tau) [1+ M_0 \tau]^{p/2} \|e_{n-1}^h\|^p
+ C_{\epsilon_2, \epsilon} \tau \|G_\tau(X(t_{n-1}))-G_\tau({\widehat Y}_{n-1}^h)\|_{\mathcal L_2^0}^p.  
\end{align*} 
By \eqref{ehn}, we obtain
\begin{align} \label{ehn-ind}
&\mathbb E_{n-1} \|e^h_n\|^p
 \le \frac{(1+\epsilon\tau)^2 [1+ M_0 \tau]^{p/2} }{[1+ 2 (\lambda_1-\epsilon) \tau]^{p/2} } \|e_{n-1}^h\|^p \nonumber \\
& \quad + C \tau [\|F_\tau(X(t_{n-1}))-F_\tau({\widehat Y}_{n-1}^h)\|_{-1}^p 
+ \|G_\tau(X(t_{n-1}))-G_\tau({\widehat Y}_{n-1}^h)\|_{\mathcal L_2^0}^p].
\end{align} 
Then we conclude \eqref{err} by taking expectations above and using the estimate \eqref{err-fgtau}.

When $M_0<2 \lambda_1$, we conclude \eqref{err} for the infinite time horizon by  induction, \eqref{ehn-ind}, \eqref{l-tau}, \eqref{err-fgtau}, and the claim: there exist constants $C_0>0$ and $\tau_4 \in (0, 1]$ s.t. 
\begin{align} \label{l-tau}
L(\tau):=\frac{(1+\epsilon\tau)^2[1+M_0\tau]^{p/2}}{[1+2(\lambda_1-\epsilon)\tau]^{p/2}}
\le 1- C_0 \tau,
\quad \forall ~ \tau \in (0, \tau_4).
\end{align} 
Indeed, the function $L$ has a Taylor expansion around the origin: 
$L(\tau) = 1 + [2 \epsilon + p(M_0-2\lambda_1 +2 \epsilon)/2] \tau + O(\tau^2)$.
Since \(M_0 < 2\lambda_1\), we may choose \(\epsilon > 0\) sufficiently small s.t. $2 \epsilon + p(M_0-2\lambda_1 +2 \epsilon)/2<0$. 
So there exist constants $c_0>0$ and $\tau_4 \in (0, 1]$ s.t. $L(\tau) \le 1-c_0 \tau$ for all $\tau \in (0, \tau_4)$.
Differentiating
$F(\tau) := \ln L(\tau) = 2\ln(1+\epsilon\tau) + p\ln(1+M_0\tau)/2 - p\ln(1+2(\lambda_1-\epsilon)\tau)/2$, we obtain
\[
F'(\tau) = \frac{2\epsilon}{1+\epsilon\tau}
- \frac{p [2(\lambda_1-\epsilon) - M_0]}{2(1+M_0\tau)(1+2(\lambda_1-\epsilon)\tau)}<0
\]
for sufficiently small \(\epsilon\) and for all \(\tau \in (0, 1)\). Hence \(F\) is strictly decreasing on \((0,1)\) with \(F(0)=0\), we get $F(\tau) < 0$, which shows $L(\tau) < 1$. 
On \([\tau_4,1]\), \(L(\tau) < 1\), so there exists \(M_1 < 1\) s.t.
$L(\tau) \le M_1$ for all $\tau \in [\tau_4,1]$.
Now choose \(C_0 =\max\{c_0, (1-M) \tau_4^{-1}\}\), which can be shown to satisfy \eqref{l-tau}. 
\end{proof}

 \begin{remark}  \label{rk-con} 
As illustrated in Example \ref{rk-ex} and Remark \ref{rk-fgtau'}, in the Lipschitz diffusion case, letting $L_g$ denote the Lipschitz constant of $g$ and $M_1$ an upper bound for $f_\tau'$, the uniform-in-time condition \eqref{con-uni}, together with $M_0<2\lambda_1$, reduces to $2 \max\{L_1, M_1\} + C_Q(p-1)L_g^2< 2\lambda_1$, which coincides with that in \cite{Liu26}.
\end{remark}

 \bibliographystyle{plain}
  \bibliography{bib.bib}
\end{document}